\documentclass[10pt, leqno]{article}

\usepackage{amsmath,amssymb,amsthm, epsfig}

\usepackage{hyperref}

\usepackage{amsmath}

\usepackage{amssymb}

\usepackage{hyperref}

\usepackage{color}

\usepackage{ulem}

\usepackage{epsfig}

\usepackage{dsfont}

\usepackage{mathrsfs}

\usepackage{mathtools}

\title{Free boundary regularity for a class of one-phase problems with non-homogeneous degeneracy.}
\author{\it by \smallskip \\
Jo\~{a}o Vitor da Silva,\footnote{\noindent \textsc{Jo\~ao Vitor da Silva}.
Universidade Estadual de Campinas - UNICAMP. Department of Mathematics. Campinas - SP, Brazil.
\texttt{E-mail address: jdasilva@unicamp.br}
}
\smallskip \qquad Giane C. Rampasso\footnote{\noindent \textsc{Giane Casari Rampasso}.
Universidade Estadual de Campinas - UNICAMP. Department of Mathematics. Campinas - SP, Brazil.
\texttt{E-mail address: girampasso@ime.unicamp.br}
}
\smallskip \qquad Gleydson C. Ricarte\footnote{\noindent \textsc{Gleydson Chaves Ricarte}.
Universidade Federal do Cear\'{a} - UFC. Department of Mathematics. Fortaleza - CE, Brazil - 60455-760.
\texttt{E-mail address: ricarte@mat.ufc.br}
}
\smallskip \\
\quad $\&$ \quad
\smallskip \\
Hern\'an A. Vivas\footnote{\noindent \textsc{Hern\'an Agust\'in Vivas}.
Instituto de C\'alculo, CABA, Argentina - Centro Marplatense de Investigaciones Matem\'aticas, Mar del Plata, Argentina.
\texttt{E-mail address: havivas@mdp.edu.ar}
}}

\newlength{\hchng}
\newlength{\vchng}
\def \R {\mathbb{R}}

\def \Div {\mathrm{div}}
\def \dist {\mathrm{dist}}

\def \Leb {\mathscr{L}^n}
\def \tr {\mathrm{tr}}

\newcommand{\defeq}{\mathrel{\mathop:}=}
\def \diam {\mathrm{diam}}

\newtheorem{theorem}{Theorem}[section]

\newtheorem{lemma}[theorem]{Lemma}

\newtheorem{corollary}[theorem]{Corollary}

\theoremstyle{definition}

\newtheorem{definition}[theorem]{Definition}

\theoremstyle{remark}

\newtheorem{remark}[theorem]{Remark}

\numberwithin{equation}{section}

\newcommand{\intav}[1]{\mathchoice {\mathop{\vrule width 6pt height 3 pt depth  -2.5pt
\kern -8pt \intop}\nolimits_{\kern -6pt#1}} {\mathop{\vrule width
5pt height 3  pt depth -2.6pt \kern -6pt \intop}\nolimits_{#1}}
{\mathop{\vrule width 5pt height 3 pt depth -2.6pt \kern -6pt
\intop}\nolimits_{#1}} {\mathop{\vrule width 5pt height 3 pt depth
-2.6pt \kern -6pt \intop}\nolimits_{#1}}}

\begin{document}
\maketitle

\begin{abstract}
%\vspace{1pc}
We consider a one-phase free boundary problem governed by doubly degenerate fully non-linear elliptic PDEs with non-zero right hand side, which should be understood as an analog (non-variational) of certain double phase functionals in the theory of non-autonomous integrals. By way of brief elucidating example, such non-linear problems in force appear in the mathematical theory of combustion, as well as in the study of some flame propagation problems. In such an environment we prove that solutions are Lipschitz continuous and they fulfil a non-degeneracy property. Furthermore, we address the Caffarelli's classification scheme: Flat and Lipschitz free boundaries are locally $C^{1, \beta}$ for some $0< \beta(\verb"universal")<1$. Particularly, our findings are new even for the toy model
$$
\mathcal{G}_{p, q}[u] \defeq \left[|\nabla u|^p + \mathfrak{a}(x)|\nabla u|^q \right] \Delta u, \quad \text{for} \quad  0<p<q< \infty \quad  \text{and} \quad  0 \le \mathfrak{a} \in C^0(\Omega).
$$
We also bring to light other interesting doubly degenerate settings where our results still work. Finally, we present some key tools in the theory of degenerate fully nonlinear PDEs, which may have their own mathematical importance and applicability.
\end{abstract}

\medskip

\noindent{\bf \scriptsize{KEYWORDS}}: Free boundary regularity theory, doubly degenerate operators, one-phase problems.

\smallskip

\noindent{\bf MSC2010}: 35B65, 35J60, 35J70, 35R35. 

%\newpage
%\tableofcontents

\section{Introduction}

In this work, we establish regularity estimates to solutions and their interfaces in a so-named one-phase free boundary problems (for short FBPs) for some very degenerate elliptic operators. Precisely, given a bounded domain $\Omega \subset \mathbb{R}^n$, we consider the doubly degenerate fully non-linear elliptic problem
\begin{eqnarray}\label{P 5.1. introduc}
	\left \{
		\begin{array}{rclcl}
			\mathcal{H}(x, \nabla u) F(x, D^2 u) & = & f(x) & \text{ in } & \Omega_{+}\left( u \right),\\
			|\nabla u| & = & \mathrm{Q}(x) & \text{ on } & \mathfrak{F}(u),
		\end{array}
	\right.
\end{eqnarray}
where $\mathcal{H}: \Omega \times \mathbb{R}^{n} \rightarrow \mathbb{R}$ and $F: \Omega \times \text{Sym}(n) \to \R$ satisfy appropriate structural assumptions (to be clarified soon), $\mathrm{Q} \geq 0$ is a continuous function, $f \in L^{\infty}\left( \Omega \right) \cap C\left( \Omega \right)$ and
$$
u\geq 0\text{ in }\Omega,\quad	\Omega^{+}(u)\defeq \{x \in \Omega : u(x) >0\} \quad \textrm{and} \quad \mathfrak{F}(u) \defeq \partial \Omega^{+}(u) \cap \Omega.
$$

One of the main signatures of the model case \eqref{P 5.1. introduc} is its interplay between two different kinds of degeneracy laws, in accordance with the values of the modulating function $\mathfrak{a}$.
$$
 \Omega \times \mathbb{R}^n \ni (x, \xi) \mapsto \mathcal{H}(x, \xi) \propto |\xi|^p + \mathfrak{a}(x)|\xi|^q \quad  0<p<q< \infty \quad  \text{and} \quad  0 \le \mathfrak{a} \in C^0(\Omega).
$$
Therefore, the diffusion process exhibit a non-uniformly elliptic and doubly degenerate feature, which mixes up two power-degenerate type operators (cf. \cite{ART15}, \cite{ART17}, \cite{BD14}, \cite{BD15}, \cite{BDL19}, \cite{daSLR21}, \cite{daSV20}, \cite{daSV21}, \cite{IS12} and \cite{IS16} for related regularity estimates and free boundary problems driven by second order operators with a single degeneracy law).

Mathematically, \eqref{P 5.1. introduc} consists of a model equation for a fully nonlinear prototype enjoying a non-homogeneous degeneracy law, which constitutes an analogous in non-divergence form of certain variational integrals from the Calculus of Variations with double phase structure:
\begin{equation}\label{DPF}\tag{\bf DPF}
	\displaystyle   \left(W_0^{1,p}(\Omega)+u_0, L^m(\Omega)\right) \ni (w, f) \mapsto \text{min} \int_{\Omega} \left(\frac{1}{p}|\nabla w|^p+\frac{\mathfrak{a}(x)}{q}|\nabla w|^q-fw\right)dx,
\end{equation}
where $\mathfrak{a}\in C^{0, \alpha}(\Omega,[0, \infty))$, for some $0< \alpha \leq 1$, $1<p\le q< \infty$ and $m \in (n, \infty]$, see \cite{BCM15}, \cite{CM15}, \cite{DeFM19} and \cite{DeFO19} and the references therein.

Let us remember that the mathematical studies for the model case given by \eqref{DPF} date back to Zhikov's fundamental works in the context of Homogenization problems and Elasticity theory, and they also represent new examples of the occurrence of Lavrentiev phenomenon, see \textit{e.g.} \cite{Zhi93} (see also \cite{BCM15} and \cite{CM15}). Moreover, minimizers to \eqref{DPF}, i.e. weak solutions of
$$
   - \Div(\mathcal{A}(x,\nabla u) \nabla u) = f(x) \qquad \text{with}\qquad \mathcal{A}(x,\xi) \defeq |\xi|^{p-2} +\mathfrak{a}(x)|\xi|^{q-2},
$$
play an important role in some contexts of Materials Science and engineering, where they describe the behavior of certain strongly anisotropic materials, whose hardening estates, connected to the gradient's growth exponents, change point-wisely. In a precise way, a mixture of two heterogeneous materials, with hardening $(p\&q)$-exponents, can be performed according to the intrinsic geometry of the null set of the modulation function $x \mapsto \mathfrak{a}(x)$.

Let us come back to our main purpose. Historically the mathematical investigation of the regularity of the free boundary $\mathfrak{F}(u)$ in problems like \eqref{P 5.1. introduc} has a large literature and it has presented wide advances in the last three decades or so. Let us summarize the \text{state-of-art} of such progresses:

\begin{enumerate}
	\item \textit{Uniform Elliptic case -  Variational approach.} The case $f = 0$ and $\mathcal{H}(x, \xi) = 1$ (for a second order linear operator), was widely studied in the Caffarelli \textit{et al}' seminal works \cite{AC81}, \cite{Caf87}, \cite{Caf89} by minimizing
	$$
	\displaystyle \mathcal{J}(u) \defeq \int_{\Omega \cap \{u>0\}} \mathfrak{f}(x, u(x), \nabla u(x))dx \longrightarrow \text{min},
	$$
	where was proved existence of a minimum and regularity of the free boundary via blow-up techniques, or via singular perturbation methods for the problem $\Delta u_{\varepsilon} = \beta_{\varepsilon}(u_{\varepsilon})$, see also Caffarelli-Salsa's nowadays classic monograph \cite{CS05}.
	
	\item \textit{Degenerate cases --  Variational approach.} In this setting, we may firstly quote the work due to Danielli and Petrosyan in \cite{DP05}, in which they established the regularity near ``flat points'' of the free boundary of non-negative solutions to the minimization problem
	$$
	\displaystyle \text{min} \mathcal{J}_p(u) \qquad \text{with} \qquad \mathcal{J}_p(u) \defeq \int_{\Omega} \left(|\nabla u|^p + \lambda_0^p\chi_{\{u>0\}}\right)dx,
	$$
	which is governed by the $p$-Laplacian operator, for $f = 0$, $1<p<\infty$ and $\lambda>0$. Further, Mart\'{i}nez and Wolanski in \cite{MW08} completely address the optimization problem of minimizing
	$$
	\displaystyle \text{min} \mathcal{J}_{\mathrm{G}}(u) \qquad \text{with} \qquad  \mathcal{J}_{\mathrm{G}}(u) \defeq \int_{\Omega} \left(\mathrm{G}(|\nabla u|) + \lambda_0\chi_{\{u>0\}}\right)dx,
	$$
	in an Orlicz-Sobolev scenario (cf. \cite{Chl18}), thereby extending the Alt-Caffarelli's theory in \cite{AC81} for such a context. In the aftermath, Fern\'{a}ndez Bonder \textit{et al} in \cite{FBMW10}, and Lederman and Wolanski in \cite{LW17}, \cite{LW19} and \cite{LW21} completed the study of existing, Lipschitz regularity and regularity of the free boundary for homogeneous/inhomogeneous free boundary problems driven by $p(x)$-Laplacian type operators as follows
	$$
	\left\{
	\begin{array}{rclcl}
		\Div(|\nabla u|^{p(x)-2} \nabla u) & = & f(x) & \text{ in } & \Omega_{+}\left( u \right),\\
		|\nabla u| & = & \lambda^{\ast}(x) & \text{ on } & \mathfrak{F}(u),
	\end{array}
	\right.
	$$
	
	\item \textit{Uniform elliptic case -- Non-variational approach.} In the context of fully non-linear elliptic equations, the homogeneous problem, i.e. $f=0$ (with $\mathcal{H}(x, \xi) \equiv 1$), we refer the reader to \cite{Fel01} and the refences therein. On the other hand, the non-homogeneous case,  $f\neq 0$ (with $\mathcal{H}(x, \xi) \equiv 1$), was studied in the series of De Silva et al's fundamental works \cite{DeS11} and \cite{DFS15} in the one and two-phase scenarios respectively.
	
\end{enumerate}

In spite of these prolific references, one phase problems are still far from being completely understood from a regularity perspective. In particular, according our scientific knowledge, there are no results in this direction for problems like \eqref{P 5.1. introduc} with $0<p < q < \infty$ (i.e. a doubly degenerate, non-variational structure). We also point out that the regularity of the free boundary for the inhomogeneous problem $f \neq 0$ has not been obtained even in the case of
$$
\mathcal{G}(u) \defeq \left[|\nabla u|^p + \mathfrak{a}(x)|\nabla u|^q \right] \tr(\mathbb{A}(x)D^2 u) \quad \text{with} \quad (\text{A2}) \quad \text{and} \quad \eqref{1.3} \quad \text{in force},
$$
where $\mathbb{A}: \Omega \to \text{Sym}(n)$ is a uniform elliptic and continuous matrix.

In this manuscript, precisely, we will study a very general form of such a problem with free boundaries, in which the diffusion process degenerates in a non-homogeneous fashion. In such a scenario, we establish optimal regularity results for solutions and fine properties of the free boundary. Furthermore, such findings are relevant from both the pure and the applied mathematics view points, as well as they may open the possibility of dealing with other interesting scenarios in elliptic regularity theory.

In turn, the FBP considered in \eqref{P 5.1. introduc} also appears as the limit of certain inhomogeneous singularly perturbed problems in the non-variational context of high energy activation model in combustion and flame propagation theories (cf. \cite{ART17}, \cite{RS15} and \cite{RT11} for related topics), whose simplest mathematical model (in this case) is given by: for each $\varepsilon>0$ fixed, we seek a non-negative profile $u^{\varepsilon}$ satisfying
$$
\left\{
\begin{array}{rclcl}
	\left[|\nabla u^{\varepsilon}|^p + \mathfrak{a}(x)|\nabla u^{\varepsilon}|^q \right] \Delta u^{\varepsilon} & = & \frac{1}{\varepsilon}\beta\left(\frac{u^{\varepsilon}}{\varepsilon}\right) + f_{\varepsilon}(x) & \mbox{in} & \Omega,\\
	u^{\varepsilon}(x) & = & g(x) & \mbox{on} & \partial \Omega,
\end{array}
\right.
$$
in the viscosity sense for suitable data $p, q \in (0, \infty)$, $\mathfrak{a}$, $g$, where $\beta_{\varepsilon}$ behaves singularly of order $\mbox{o} \left(\varepsilon^{-1} \right)$ near $\varepsilon$-level surfaces. In such a scenario, existing solutions are locally (uniformly) Lipschitz continuous (see, \cite[Theorem 1.4]{daSJR21}). Thus, up to a subsequence, there exists a function $u_0$, obtained as the uniform limit of $u^{\varepsilon_{k_j}}$, as $\varepsilon_{k_j} \to 0$. Furthermore, such a limiting profile satisfies
\[
\left[|\nabla u_0|^p + \mathfrak{a}(x)|\nabla u_0|^q \right] \Delta u_0 = f_0(x)
\]
for an appropriate RHS $f_0\geq 0$ in the viscosity sense (see also \cite{RST17} for a non-variational one-phase problem driven by a strongly degenerate operator).

Therefore, we are able to apply our results to limit functions of such inhomogeneous singular perturbation problems for the doubly degenerate fully nonlinear operator that we have addressed in \cite{daSJR21}.

It is worthwhile to highlight that we can consider/recover fully nonlinear models of degenerate/singular type as follow
$$
\mathcal{G}_p(x, \xi, \mathrm{X}) \defeq |\xi|^{p}F(x, \mathrm{X}) \quad (\text{with (A0)-(A2) in force})
$$
in any Euclidian dimension and such that $-1<p< \infty$ (cf. \cite{LR18I}). Particularly, we may apply our findings to limit profiles of inhomogeneous singular perturbation problems for fully nonlinear operators of degenerate type, which they were addressed in \cite{ART17} by third author \text{et al}.

\subsection{Assumptions and statement of the main results}

Next, we will present the structural assumptions to be made throughout this work.

\begin{itemize}
  \item[(A0)]({{\bf Continuity and normalization condition}})
  $$
  \text{Fixed}\,\, \Omega \ni x \mapsto F(x, \cdot) \in C^{0}(\text{Sym}(n)) \quad  \text{and} \quad F(\cdot, \text{O}_n) = 0
  $$
  where $\text{O}_n$ is the zero matrix; this normalizing assumption can be impose without loss of generality.

  \item[(A1)]({{\bf Uniform ellipticity}}) For any pair of matrices $\mathrm{X}, \mathrm{Y}\in Sym(n)$
\[
\mathscr{P}^{-}_{\lambda,\Lambda}(\mathrm{X}-\mathrm{Y})\leq F(x, \mathrm{X})-F(x, \mathrm{Y})\leq \mathscr{P}^{+}_{\lambda,\Lambda}(\mathrm{X}-\mathrm{Y})
\]
where $\mathscr{P}^{\pm}_{\lambda,\Lambda}$ stand for \textit{Pucci's extremal operators} given by
$$
   \mathscr{P}_{\lambda, \Lambda}^-(\mathrm{X}) \defeq \lambda\sum_{e_i>0}e_i(\mathrm{X})+\Lambda\sum_{e_i<0}e_i(\mathrm{X})\quad\textrm{ and }\quad \mathscr{P}_{\lambda, \Lambda}^+(X)\defeq \Lambda\sum_{e_i>0}e_i(\mathrm{X})+\lambda\sum_{e_i<0}e_i(\mathrm{X})
$$
for \textit{ellipticity constants} $0<\lambda\leq \Lambda< \infty$, where $\{e_i(\mathrm{X})\}_{i}$ are the eigenvalues of $\mathrm{X}$.

Moreover, for our Lipschitz estimates, we must require some sort of uniform continuity assumption on the coefficients:

  \item[(A2)]({{\bf $\omega-$continuity of coefficients}}) There exist a uniform modulus of continuity $\omega: [0, \infty) \to [0, \infty)$ and a constant $\mathrm{C}_{\mathrm{F}}>0$ such that
$$
\Omega \ni x, x_0 \mapsto \Theta_{\mathrm{F}}(x, x_0) \defeq \sup_{\mathrm{X} \in  Sym(n) \atop{\mathrm{X} \neq 0}}\frac{|F(x, \mathrm{X})-F(x_0, \mathrm{X})|}{\|\mathrm{X}\|}\leq \mathrm{C}_{\mathrm{F}}\omega(|x-x_0|),
$$
which measures the oscillation of coefficients of $F$ around $x_0$. We will denote $\Theta_{\mathrm{F}}(x,0)$ simply by $\Theta_{\mathrm{F}}(x)$. Finally, we define
$$
\|F\|_{C^{\omega}(\Omega)} \defeq \inf\left\{\mathrm{C}_{\mathrm{F}}>0: \frac{\Theta_{\mathrm{F}}(x, x_0)}{\omega(|x-x_0|)} \leq \mathrm{C}_{\mathrm{F}}, \,\,\, \forall \,\,x, x_0 \in \Omega, \,\,x \neq x_0\right\}.
$$
\end{itemize}

In our studies, the diffusion properties of the model \eqref{P 5.1. introduc} degenerate along an \textit{a priori} unknown set of singular points of existing solutions:
$$
   \mathcal{S}(u, \Omega) \defeq \{x \in \Omega: |\nabla u(x)| = 0\}.
$$
For this reason, we will enforce that $\mathcal{H}:\Omega \times \R^n \to [0, \infty)$ behaves as
\begin{equation}\label{1.2}
     L_1 \cdot \mathcal{K}_{p, q, \mathfrak{a}}(x, |\xi|) \leq \mathcal{H}(x, \xi)\leq L_2 \cdot \mathcal{K}_{p, q, \mathfrak{a}}(x, |\xi|)
\end{equation}
for constants $0<L_1\le L_2< \infty$, where
\begin{equation}\label{N-HDeg}\tag{\bf N-HDeg}
   \mathcal{K}_{p, q, \mathfrak{a}}(x, |\xi|) \defeq |\xi|^p+\mathfrak{a}(x)|\xi|^q, \,\,\,\text{for}\,\,\, (x, \xi) \in \Omega \times \R^n.
\end{equation}

In addition, for the non-homogeneous degeneracy \eqref{N-HDeg}, we suppose that the exponents $p, q$ and the modulating function $\mathfrak{a}(\cdot)$ fulfil
\begin{equation}\label{1.3}
   0< p \le q< \infty \qquad \text{and} \qquad \mathfrak{a} \in C^0(\Omega, [0, \infty)).
\end{equation}

Finally, we will assume the following condition: there exist a universal constant $\mathrm{C}_{\mathfrak{a}}>0$ and a modulus of continuity $\omega_{\mathfrak{a}}: [0, \infty) \to [0, \infty)$ such that
\begin{equation}\label{Cont-H}
   |\mathcal{H}(x, \xi)-\mathcal{H}(y, \xi)| \le \mathrm{C}_{\mathfrak{a}}\omega_{\mathfrak{a}}(|x-y|)|\xi|^q  \qquad \forall\,\,\,(x, y, \xi) \in \Omega \times \Omega\times \R^n.
\end{equation}

Now, we can start stating our main results. In a first point, we will establish optimal Lipschitz regularity to solutions of \eqref{P 5.1. introduc} (see Section \ref{section2} for the definition of viscosity solutions).

\begin{theorem}[{\bf Optimal Lipschitz regularity}] \label{Lipschitz}
Let $\mathrm{Q} \in C^0(B_1; [0, \infty)) \cap L^{\infty}(B_1; [0, \infty))$ and $u$ be a bounded viscosity solution to (\ref{P 5.1. introduc}) in $B_1$. Then, there exists a universal constant $\mathrm{C_1} = \mathrm{C_1}(n, \lambda, \Lambda, \mathfrak{a}, L_1, p, q)>0$ such that
\begin{equation}\label{estimate1-Lipschitz}
	u(x_0) \leq \mathrm{C_1}.\left(\|u\|_{L^{\infty}(B_1)}+\|\mathrm{Q}\|_{L^{\infty}(B_1)} + \max\left\{\left\|f\right\|_{L^{\infty}(B_1)}^{\frac{1}{p+1}}, \left\|f\right\|_{L^{\infty}(B_1)}^{\frac{1}{q+1}}\right\}\right)\dist(x_0, \mathfrak{F}(u)),
\end{equation}
for all $x_0 \in B_{1/2}$; i.e., solutions have at most linear behavior close to free boundary points.
Particularly, there exists $C_2=C_2(n, \lambda, \Lambda, L_1, p, q, \|F\|_{\mathcal{C}^\omega})>0$ such that
\begin{equation}\label {estimate2-Lipschitz}
\|\nabla u\|_{L^{\infty}(B_{1/2})} \leq \mathrm{C_2}.\left(\|u\|_{L^{\infty}(B_1)}+\|\mathrm{Q}\|_{L^{\infty}(B_1)} + \max\left\{\left\|f\right\|_{L^{\infty}(B_1)}^{\frac{1}{p+1}}, \left\|f\right\|_{L^{\infty}(B_1)}^{\frac{1}{q+1}}\right\}+1\right).
\end{equation}
\end{theorem}

Once the (optimal) growth control away from the free boundary is obtained, the next point of interest is the \textit{non-degeneracy} of solutions, that controls the behavior from below:

\begin{theorem}[{\bf Non-degeneracy of solutions}] \label{Nondeg}
Let $\mathrm{Q} \in C^0(B_1; [0, \infty)) \cap L^{\infty}(B_1; [0, \infty))$ and $u$ be a bounded viscosity solution to (\ref{P 5.1. introduc}) in $B_1$. Further, suppose that $\mathfrak{F}(u)$ is a Lipschitz graph in $B_1$ with $\mathfrak{F}(u) \cap B_{1/2}^{+}(u) \neq \emptyset$. There exists a universal $\eta_0\in(0,1)$ a universal constant $\mathrm{C}_{\ast} = \mathrm{C}(n, \lambda, \Lambda, p, q, \|F\|_{C^{\omega}(B_1)})>0$ such that if
\[
\|\mathrm{Q}-1\|_{L^{\infty}(B_1)}<\eta_0
\]
then
$$
u(x_0) \geq \mathrm{C}_{\ast}.\dist(x_0, \mathfrak{F}(u)),
$$
for all $x_0 \in B^{+}_{1/2}(u)$;  i.e. solutions growth at least in a linear fashion close to free boundary points.
\end{theorem}

We will also develop the regularity theory of $\mathfrak{F}(u)$. Precisely, we will adapt the technique presented in \cite{DeS11} to prove that flat free boundaries are $C^{1, \beta}$:

\begin{theorem}[{\bf Flatness implies $C^{1,\beta}$}]
\label{th 5.1. introd.}
Let $u$ be a viscosity solution to (\ref{P 5.1. introduc}) in $B_1$. Suppose that $0 \in \mathfrak{F}\left( u \right),\:\mathrm{Q}\left(0 \right)= 1$ and $F(0,X)$ is uniformly elliptic. Then, there exists a constant $\tilde{\varepsilon}(\verb"universal") > 0$ such that, if the graph of $u$ is $\tilde{\varepsilon}$-flat in $B_1$, i.e.
\begin{eqnarray*}
\label{t 5.1.1. introd.}
\left(x_{n} - \tilde{\varepsilon}\right)^{+} \leq u \left( x \right) \leq \left( x_{n} + \tilde{\varepsilon} \right)^{+} \ \ \mbox{for} \ \ x \in B_{1},
\end{eqnarray*}
and
\begin{eqnarray*}
\label{t 5.1.2. introd.}
\max\left\{\Vert f \Vert_{L^{\infty}\left( B_{1} \right)}, \ \ \left[ \mathrm{Q} \right]_{C^{0,\alpha}\left( B_{1}\right)}, \|F\|_{C^\omega(B_1)}\right\} \leq \tilde{\varepsilon},
\end{eqnarray*}
then $\mathfrak{F}\left( u \right)$ is $C^{1,\beta}$ in $B_{1/2}$ for some (universal) $\beta\in(0,1)$.
\end{theorem}

Finally, through a blow-up from Theorem \ref{th 5.1. introd.} and the approach used in \cite{Caf87}, we obtain our last main result:

\begin{theorem}[{\bf Lipschitz implies $C^{1,\beta}$}] \label{Holder1}
Let $u$ be a viscosity solution for the free boundary problem \eqref{P 5.1. introduc}. Assume further that $0 \in \mathfrak{F}(u)$, $f \in L^{\infty}(B_1)$ is continuous in $B^{+}_{1}(u)$ and $\mathrm{Q}(0)>0$. If $\mathfrak{F}(u)$ is a Lipschitz graph in a neighborhood of $0$, then $\mathfrak{F}(u)$ is $C^{1,\beta}$ in a (smaller) neighborhood of $0$.
\end{theorem}

In Theorem \ref{Holder1}, the size of the neighborhood where $\mathfrak{F}(u)$ is $C^{1,\beta}$ depends on the radius $r$ of the ball $B_r$ where $\mathfrak{F}(u)$ is Lipschitz, the Lipschitz norm of $\mathfrak{F}(u)$, $n$ and $\|f\|_{\infty}$. We also emphasize that to obtain the Theorems \ref{th 5.1. introd.} and \ref{Holder1} via the \textit{improvement of flatness property} for the graph of $u$, we will need a version of Hopf type estimate, Harnack inequality, Lipschitz regularity, and Non-degeneracy for $u$. We stress that all of these key tools will be developed either in the first part or the Appendix of our manuscript and they have their own independent mathematical relevance.

\subsection{Major obstacles and strategy for the free boundary regularity}

Let us comment on the main obstacles we came across in order to obtain an \textit{improvement of flatness} property for the graph of a solution of \eqref{P 5.1. introduc} and how to overcome them.

As stated in \cite{DeS11}, the strategy of proving Theorem \ref{th 5.1. introd.} is to obtain an \textbf{improvement of flatness} property for the graph of a solution $u$: if the graph of $u$ oscillates away $\varepsilon$ from a hyperplane in $B_{1}$ then in $B_{\delta_{0}}$ it oscillates $\frac{\delta_{0}\varepsilon}{2}$ away from possibly a different hyperplane. We stress that fundamental tools to achieve this property are a Harnack type inequality and characterizing of limiting solutions; By way of information, the structure of the operator $\mathcal{G}_{p, q}[u] \defeq \mathcal{H}(x,\nabla u)F(x,D^2u)$ requires some non-trivial adaptations.

\begin{enumerate}
\item \textbf{Harnack type inequality}. When we consider the problem \eqref{P 5.1. introduc} for $0 < p \le q < \infty$, the first difficulty we find lies in the following fact: in general, if $\ell$ is an affine function and $u$ is a solution to the problem
\begin{eqnarray}
\label{Degen. prob.}
 \mathcal{H}(x,\nabla u) F(x,D^2 u)  =  f(x) \quad \text{ in } B_{r}(x_{0}), \quad \text{where} \ x_{0} = \frac{e_{n}}{10},
\end{eqnarray}
we can not conclude that $u + \ell$ is a solution to the equation \eqref{Degen. prob.} yet. In contrast, for $p=q = 0$ we know $u + \ell$ is still solution for \eqref{Degen. prob.}. In effect, in \cite{DeS11}, De Silva have used this fact, thereby allowing us to apply the Harnack inequality for $v(x) = u(x) - x_{n}$, which play a crucial role in reaching an \textit{improvement of flatness} for the graph of $u$. We will overcome this difficulty as follows: \\

{\bf Step 1}. We notice that the function $v(x) = u(x) - x_{n}$ is a solution to the problem
\begin{eqnarray}
\label{Degen. with e}
 \mathcal{H} (x, \nabla v + e_n)F(x,D^2 v)  =  f(x) \quad \text{ in } B_{r}(x_{0}),
\end{eqnarray}
Then, we know (see Appendix) that $v$ satisfies the following Harnack Inequality
\begin{eqnarray}
\label{Imbert Harnack 1}
    \sup_{B_{r/2}(x_{0})}v \leq \mathrm{C} \cdot \left\{ \inf_{B_{r/2}} \, v +  (q+1)^{\frac{1}{q+1}} \max\left\{r^{\frac{p+2}{p+1}}, r^{\frac{q+2}{p+1}}\right\} \Pi^{f,\mathfrak{a}}_{p,q}  \right\},
\end{eqnarray}
where $\mathrm{C}(\verb"universal")>0$. \\

{\bf Step 2}. Since we will make use of a blow-up procedure to prove our results (Theorem \eqref{th 5.1. introd.} and Theorem \eqref{Holder1}), we may assume WLOG that $\Vert f \Vert_{\infty} $ is small. Hence, we can consider the scaled function $v_{r}(x) = \frac{v(rx + x_{0})}{r}$ and apply \eqref{Imbert Harnack 1} to get
\begin{eqnarray}
\label{Scaling Imbert Harnack}
   \sup_{B_{r/2}(x_{0})}v \leq C\left\lbrace \inf_{B_{r/2}(x_{0})} v  +  \max\left\{r^{\frac{p+2}{p+1}}, r^{\frac{q+2}{p+1}}\right\}  \right\rbrace,
\end{eqnarray}
for a constant $C=C(n,p,q, \mathfrak{a}, \lambda,\Lambda)>0$.

{\bf Step 3}. Notice that the Harnack Inequality \eqref{Scaling Imbert Harnack} is slightly different from the one addressed in \cite{DeS11}. In effect, for $\varepsilon \in (0, 1)$, De Silva have used the inequality
\begin{eqnarray}
\label{DeSilva Harnack}
   \sup_{B_{r/2}(x_{0})}v \leq C\left\lbrace \inf_{B_{r/2}(x_{0})} v  + \Vert f \Vert_{L^\infty(B_{r/2}(x_{0}))}  \right\rbrace
\end{eqnarray}
to show that if $\Vert f \Vert_{\infty}$ satisfies the \textit{smallness} assumption $\Vert f \Vert_{L^\infty(B_{r/2}(x_{0}))} \leq \varepsilon^{2}$, then we are able to build radial barriers $w_{r, x_{0}}$ and apply comparison techniques to achieve an appropriate Harnack type inequality to establish the desired \textit{improvement flatness} (see \cite[Theorem 3.1 and Lemma 3.3]{DeS11} for more details). A careful analysis of the behavior of $v = u - x_{n}$ (or $v = x_{n} - u$) in a ball $B_{r_{1}}(x_{0})$ with
$$\vert \nabla u \vert < \frac{1}{2} \quad \text{in} \ B_{r_{1}}(x_{0}), $$
and $r_{1} = r_{1}(\mu) >0$, reveals that if we consider radial barriers $w_{r, x_{0} + r_{2}e_{n}}$ the condition $\Vert f \Vert_{\infty} \leq \varepsilon^{2}$ used in \eqref{DeSilva Harnack} can be replaced by an adequate \textit{smallness} condition of the \textit{radius} $r =r(r_{2})$ in \eqref{Scaling Imbert Harnack} to obtain a Harnack type inequality, where $r_{2} = r_{2}(r_{1})$.
\end{enumerate}

\subsection{Some extensions and further comments}

In conclusion, we stress that our approach is particularly refined and quite far-reaching in order to be employed in other classes of operators. As a matter of fact, we can also extended our results for nonlinear elliptic equations with non-homogeneous term as follows:

\begin{enumerate}

\item {\bf Multi degenerate operators in non-divergence form.}

We stress that an extension of our results also holds to general multi-degenerate fully nonlinear models given by
$$
\mathcal{G}(x, D u, D^2 u) \defeq \left(|Du|^p + \sum_{i=1}^{N} \mathfrak{a}_i(x)|Du|^{q_i}\right)F(x, D^2 u) \quad (\text{with (A0)-(A2) in force}),
$$
where $0\le\mathfrak{a}_i \in C^0(\Omega)$, $i \in \{1, \cdots, N\}$, and $0<p\leq q_1\leq \cdots\leq q_N< \infty$, which are a natural non-variational counterpart of certain multi-phase variational problems treated in \cite{DeFO19} (see, \cite[pag. 8]{daSR20} for related discussion).

  \item {\bf Doubly degenerate $(p, q)-$Laplacian in non-divergence form.}

We would also like to highlight that other interesting class of degenerate operators where our results work out is the double degenerate $p-$Laplacian type operators, in non-divergence form (cf. \cite{APR17} and \cite[Section 5.1]{daSR20} for related regularity aspects), for $2<p_0\leq q_0< \infty$ and $1<p< \infty$:
     $$
         \mathcal{G}_{p_0, q_0}(x, \xi, X) = \mathcal{H}_{p_0, q_0}(x, \xi)F_p(\xi, X)
     $$
where
   $$
     \mathcal{H}_{p_0, q_0}(x, \xi)\defeq |\xi|^{p_0-2}+ \mathfrak{a}(x)|\xi|^{q_0-2} \quad (\text{with \eqref{1.3} in force})
   $$
and
$$
     F_p(\xi, X) \defeq  \tr\left[\left(\textrm{Id}_n+(p-2)\frac{\xi\otimes \xi}{|\xi|^2}\right)X\right]
$$
is the well-known Normalized $p-$Laplacian operator. Notice that $F_p$ satisfies assumption (A0)-(A2) with
$$
   \lambda_p = \min\{p-1, 1\} \quad \mbox{and}\ \quad \Lambda_p = \max\{p-1, 1\} \quad \text{and} \quad \omega \equiv 0.
$$

\item {\bf Fully nonlinear models with non-standard growth.}

 We would like to stress the class of variable-exponent, degenerate elliptic equations in non-divergence form, which is, in some extent, the non-variational counterpart of certain non-homogeneous functionals satisfying nonstandard growth conditions (see \cite{BPRT20} for an enlightening essay). Particularly, such models encompass problems ruled by the $p(x)-$Laplacian operator (cf. \cite{LW17}, \cite{LW19} and \cite{LW21}).

An archetypical example we have in mind concerns models of the form
$$
\mathcal{G}_{p(x), q(x)}(x, \xi, X) \defeq \left(|\xi|^{p(x)}+\mathfrak{a}(x)|\xi|^{q(x)}\right)F(x, \mathrm{X}) \quad (\text{with (A0)-(A2) and \eqref{1.3} in force}),
$$
for rather general exponents $p, q \in C^0(\Omega;(0, \infty)$ (see \cite{BPRT20} for details).
\end{enumerate}

The rest of the paper is organized as follows. In Section \ref{section2} we define the notion of viscosity solution to the free boundary problem \eqref{P 5.1. introduc} and gather a few tools that we shall use in the proofs of Theorems \ref{Lipschitz} and Theorem \ref{Nondeg}, which are the contents of Section \ref{Lipandnondeg}. In Section \ref{harnack} we present the proof of Harnack type inequality which in turn is used in Section \ref{improvement} to prove the improvement of flatness result. In Section \ref{reg_fron_livre} we establish the regularity of the free boundary $\mathfrak{F}(u)$, i.e. Theorems \ref{th 5.1. introd.} and \ref{Holder1}. Finally, in the Appendix we preset a discussion of the Harnack inequality, which were of use throughout the paper.

\section{Preliminaries and some auxiliary results} \label{section2}

We start by giving the definition of viscosity solution for problems of the form \eqref{P 5.1. introduc}. First, recall that given two continuous functions $u$ and $\phi$ defined in an open set $\Omega$ and a point $x_0 \in \Omega$, we say that $\phi$ touches $u$ by below (resp. above) at $x_0$ whenever
$
    u(x_0)=\phi(x_0)
$
$$
    u(x) \ge \phi(x) \,\, (\textrm{resp.} \, \, u(x) \le \phi(x)) \,\,\,\textrm{in a neighborhood} \,\,\mathcal{O} \,\, \textrm{of} \,\, x_0.
$$
If this inequality is strict in $\mathcal{O} \setminus \{x_0\}$, we say that $\phi$ touches $u$ strictly by below (resp. above).

\begin{definition}\label{d 5.3}
Let $u \in C(\Omega)$ nonnegative. We say that $u$ is a viscosity supersolution (resp.subsolution) to
\begin{eqnarray*}
	\left \{
		\begin{array}{rclcl}
			\mathcal{H}(x,\nabla u)F(x,D^2u) & = & f(x) & \text{ in } & \Omega_{+}\left( u \right),\\
			|\nabla u| & = &\mathrm{Q}(x) & \text{ on } & \mathfrak{F}(u).
		\end{array}
	\right.
\end{eqnarray*}
 if and only if the following conditions are satisfied:
\begin{enumerate}
\item[(F1)] If $\phi \in C^{2}(\Omega^{+}(u))$ touches $u$ by below (resp. above) at $x_0 \in \Omega^{+}(u)$ then
$$
   \mathcal{H}(x_0,\nabla \phi(x_0))F(x_0,D^2\phi(x_0)) \le f(x_0) \qquad \left(\textrm{resp.} \,\, \mathcal{H}(x_0,\nabla \phi(x_0))F(x_0,D^2\phi(x_0) \ge f(x_0)\right).
$$
\item[(F2)] If $\phi \in C^2(\Omega)$ and $\phi$ touches $u$ below (resp. above) at $x_0 \in \mathfrak{F}(u)$ and $|\nabla \phi|(x_0) \not= 0$ then
$$
    |\nabla \phi|(x_0) \le Q(x_0) \qquad \left( \textrm{resp.} \,\,\, |\nabla \phi|(x_0) \ge Q(x_0)\right).
$$
\end{enumerate}

We say that $u$ is a viscosity solution if it is a viscosity supersolution and a viscosity subsolution.
\end{definition}

We will further need the notion of comparison subsolution/supersolution:
\begin{definition}\label{d 5.3e}
We say $u \in C(\Omega)$ is a strict comparison subsolution (resp. supersolution) to
\begin{eqnarray*}
	\left \{
		\begin{array}{rclcl}
			\mathcal{H}(x,\nabla u)F(x,D^2u) & = & f(x) & \text{ in } & \Omega_{+}\left( u \right),\\
			|\nabla u| & = &\mathrm{Q}(x) & \text{ on } & \mathfrak{F}(u).
		\end{array}
	\right.
\end{eqnarray*}
if only if $u \in C^{2}(\overline{\Omega^{+}(u)})$ and the following conditions are satisfied:
\end{definition}
 \begin{enumerate}
\item[(G1)] $\mathcal{H}(x,\nabla u)F(x,D^2u) > f(x) \quad \text{in} \quad \Omega^{+}(u)\qquad \left(\textrm{resp.} \quad \mathcal{H}(x,\nabla u)F(x,D^2u)  < f(x) \right)$;
\item[(G2)] If $x_0 \in \mathfrak{F}(u)$, then
$$
    |\nabla u|(x_0) > \mathrm{Q}(x_0) \qquad \left(\textrm{resp.} \,\,\, 0 < |\nabla u|(x_0) < \mathrm{Q}(x_0)\right).
$$
\end{enumerate}

In the sequel, let us remember the following notion of convergence of sets:
\begin{definition}\label{Def-Hausdorff-Dist}
A sequence of sets $\{\mathfrak{A}_k\}$ is said to converge (locally) to a set $\mathfrak{A}$ in the Hausdorff distance if, given a compact set $\mathrm{K}$ and a $\delta >0$, there exist a $k=k(\delta, \mathrm{K}) \in \mathbb{N}$, the following inclusions hold:
$$
    \mathrm{K} \cap \mathfrak{A}_k \subset \mathcal{N}_{\delta}(\mathfrak{A}) \cap \mathrm{K} \quad \textrm{and} \quad \mathrm{K} \cap \mathfrak{A} \subset \mathcal{N}_{\delta}(\mathfrak{A}_k) \cap \mathrm{K},
$$
where $\mathcal{N}_{\delta}(\mathrm{E}) \defeq  \{x \in \mathbb{R}^n : \dist(x, \mathrm{E}) < \delta\}$.
\end{definition}

Next Lemma provides a crucial comparison device for solutions to FBP \eqref{P 5.1. introduc}.

\begin{lemma}\label{l 5.1} Let $u,v$ be respectively a solution and a strict subsolution to \eqref{P 5.1. introduc} in $\Omega$. If $u \ge v^{+}$ in $\Omega$ then $u > v^{+}$ in $\Omega^{+}(v) \cup \mathfrak{F}(v)$.
 \end{lemma}

As stayed in \cite{DeS11} other crucial piece of information in proving Theorem \ref{th 5.1. introd.} is the regularity of solutions to the homogeneous problem with Neumann boundary condition:
\begin{equation}\label{const}
\left\{
\begin{array}{rclcl}
 \mathrm{F}(D^2 u_{\infty}) & = & 0 & \mbox{in}  & B^{+}_{\rho_0}\\
\frac{\partial u_{\infty}}{\partial x_n} & = & g(x)  & \mbox{on} & \Upsilon_{\rho_0},
\end{array}
\right.
\end{equation}
which arises from a specific blowing-up procedure, where we denote by
$$
	B^{+}_{\rho_0}  \defeq  \left\{(x, x_n) \in \mathbb{R}^{n-1}\times \mathbb{R} : |x| <\rho_0 \,\,\,\text{and}\,\,\, x_n >0\right\}
$$
and
$$
\Upsilon_{\rho_0}  \defeq  \left\{(x^{\prime}, x_n) \in \mathbb{R}^{n-1}\times \mathbb{R} : |x^{\prime}| <\rho_0 \,\,\,\text{and} \,\,\,x_n=0\right\}.
$$

Let us present, in the sequel, the notion of viscosity solution employed in \eqref{const}.

\begin{definition}\label{]defVisc}
We say that $u_{\infty} \in C^0(B_{\rho} \cap \{x_n \ge 0\})$ is a viscosity solution to \eqref{const} if given $\mathrm{P}$ a quadratic polynomial touching $u_{\infty}$ by below (resp. by above) at $x_0 \in B_{\rho} \cap \{x_n \ge 0\}$, then
\begin{enumerate}
\item[(i)] If $ x_0 \in B^{+}_{\rho} $ then $\mathrm{F}(D^2 \mathrm{P}(x_0)) \le 0$ \,\,\,(resp. $\mathrm{F}(D^2 \mathrm{P}(x_0)) \ge 0$);
\item[(ii)] If $ x_0 \in \Upsilon_{\rho}$ then $\frac{\partial \mathrm{P}(x_0)}{\partial x_n} \le g(x_0)$ \,\,\,(resp. $\frac{\partial \mathrm{P}(x_0)}{\partial x_n} \ge g(x_0)$)
\end{enumerate}
\end{definition}

It is worth to highlight that we may choose, in the above definition, polynomials $\mathrm{P}$ touching $u_{\infty}$ strictly by below/above. Furthermore, it suffices to check sentence (ii) does hold for polynomials $\tilde{\mathrm{P}}$ such that $\mathrm{F}( D^2 \tilde{\mathrm{P}}) >0$ (see \cite{DeS11} for such details).

The following regularity estimate for solutions of \eqref{const} will play a key role in the improvement of flatness process. It holds as a consequence of the boundary regularity results addressed by Milakis-Silvestre in \cite[Theorem 6.1]{MS06}.
\begin{lemma}\label{reg_Finfty}
Let $u$ be a viscosity solution to
$$
\left\{
\begin{array}{rclcl}
F(D^2 u) &=& 0 & \mbox{in} &  B^{+}_{\frac{1}{2}} \\
\frac{\partial u}{\partial \nu}&=& g(x) & \mbox{on} & \Upsilon_{\frac{1}{2}}
\end{array}
\right.
$$
with $\|u\|_{L^{\infty}\left(B^{+}_{\frac{1}{2}}\right)} \le 1$. Then, there exist universal constants $\alpha\in(0,1)$ and $\mathrm{C}_0 >0$ such that
$$
  \displaystyle  \sup_{B^{+}_{\rho}} \frac{|u(x)-u(0) - \nabla u(0) \cdot x|}{ \rho^{1+\alpha}} \le \mathrm{C}_0 \quad \text{for}\quad  \rho \in \left(0, \frac{1}{2}\right).
$$
\end{lemma}

In this last part we will collect some fundamental auxiliary results for our purposes. The first one is a well known result so called ABP estimate. We recommend the reader to refer \cite[Theorem 8.6]{daSJR21} for an exact proof.

\begin{theorem}[{\bf Alexandroff-Bakelman-Pucci estimate}]\label{ABPthm}
  Assume that assumptions (A0)-(A2) there hold. Then, there exists $C = C(n, \lambda, p, q, \diam(\Omega))>0$ such that for any $u \in C^0(\overline{\Omega})$ viscosity solution
   $$
		\mathcal{H}(x, \nabla u) F(x, D^2 u) = f(x) \quad  \text{in} \quad  \Omega
   $$
    satisfies
$$
\displaystyle \|u\|_{L^{\infty}(\Omega)} \leq \|u\|_{L^{\infty}(\partial \Omega)} +C\cdot\diam(\Omega)\max\left\{\left\|\frac{f}{1+\mathfrak{a}}\right\|^{\frac{1}{p+1}}_{L^n(\Omega)}, \left\|\frac{f}{1+\mathfrak{a}}\right\|^{\frac{1}{q+1}}_{L^n(\Omega)}\right\}
$$
\end{theorem}

We close this section with the following $\mathcal{C}^{1,\beta}_{loc}$ regularity result for doubly degenerate fully non-linear elliptic problems.

\begin{theorem}[{\bf Gradient estimates at interior points}]\label{main1} Assume that assumptions (A0)-(A2), \eqref{1.2} and \eqref{1.3} there hold. Let $u$ be a bounded viscosity solution to
$$
 \mathcal{H}(x, \nabla u) F(x, D^2 u) = f(x) \quad \text{in} \quad \Omega
$$
 with  $f \in L^\infty(\Omega)$. Then, $u$ is $C^{1, \beta}$, at interior points, for $\beta \in (0,  \alpha_{\mathrm{Hom}}) \cap \left(0, \frac{1}{p+1}\right]$. More precisely, for any point $x_0 \in \Omega^{\prime}\Subset \Omega$ there holds
$$
       \displaystyle [u]_{C^{1, \beta}(B_r(x_0))}\leq C\cdot\left(\|u\|_{L^{\infty}(\Omega)} +1+ \|f\|_{L^{\infty}(\Omega)}^{\frac{1}{p+1}}\right),
$$
for $0<r < \frac{1}{2}$ where $C>0$ is a universal constant\footnote{A constant is said to be universal if it depends only on dimension, degeneracy and ellipticity constants, $\alpha_{\text{Hom}}$, $\beta$, $L_1, L_2$ and $\|F\|_{C^{\omega}(\Omega)}$.}.
\end{theorem}

For a proof of Theorem \ref{main1}, we refer the reader to \cite[Theorem 1.1]{daSR20}  and \cite{DeF20}.

%%%%%%%%%%%%%%%%%%%%%%%%%%%%%%%%%%%%%%%%%%%%%%%%%%%%%%%%%%%%%%%%%%%%%%%%%%%%%%%%%%%%%%%%%%%%%%%%%%%%%%%%%%%%%%%

\section{Lipschitz regularity and Non-degeneracy of solutions}\label{Lipandnondeg}

In this point, we are in a position to prove the optimal Lipschitz regularity in Theorem \ref{Lipschitz}. Nevertheless, in contrast with \cite[Lemma 6.1]{DeS11}, the proof can be obtained employing some ideas as ones in \cite{ART17}, \cite{daSJR21}, \cite{RS15}, \cite{RST17} and \cite{RT11} performed for the scenario of singularly perturbed FBPs.

\begin{proof}[{\bf Proof of Theorem \ref{Lipschitz}}]

Let $x_0 \in B_{1/2}$ such that $x_0\in B_{1/2}^+(u)$. Then, we define
	$$
	d_0\defeq \dist(x_0, \mathfrak{F}(u)).
	$$
We will suppose that $d_0 \leq \frac{1}{2}$.

Let us consider the scaled function $v_{x_0, d_0}: B_{1} \to \mathbb{R}$ defined by
 $$
    v_{x_0, d_0}(x) \defeq \frac{u(x_0+rd_0x)}{d_0},
 $$
for $r \in (0, 1)$ to be chosen later. In this point, it will be enough to prove that $v_{x_0, d_0}(0) \leq \mathrm{C}_0$ for some constant $\mathrm{C}_0(\verb"universal")>0$.

Indeed, notice that $v_{x_0, d_0}$ is a non-negative viscosity solution of
$$
     \mathcal{H}_{x_0, d_0}(x,\nabla v_{x_0, d_0})F_{x_0, d_0}(x, D^{2}v_{x_0, d_0})=f_{x_0, d_0}(x) \quad \text{in} \quad B_1
$$
where
$$
\left\{
\begin{array}{rcl}
  F_{x_0, d_0}(x, \mathrm{X}) & \defeq & r^2d_0F\left(x_0+rd_0 x,\frac{1}{r^2d_0}\mathrm{X}\right) \\
  \mathcal{H}_{x_0, d_0}(x, \xi) & \defeq &  r^p\mathcal{H}\left(x_0+rd_0x, \frac{1}{r}\xi\right)\\
  \mathfrak{a}_{x_0, d_0}(x) & \defeq & r^{p-q}\mathfrak{a}(x_0+rd_0x)\\
  f_{x_0, d_0}(x) & \defeq & r^{p+2}d_0f(x_0 + rd_0x)\\
  \mathrm{Q}_{x_0,d_0}(x)&\defeq& r\mathrm{Q}(x_0+rd_0x).
\end{array}
\right.
$$
Furthermore, $F_{x_0, d_0}, \mathcal{H}_{x_0, d_0}$ and $\mathfrak{a}_{x_0, d_0}$ satisfy the structural assumptions (A0)-(A2), \eqref{1.2} and \eqref{1.3}.

Now, let us consider the annulus $\mathcal{A}_{\frac{1}{2}, 1}\defeq B_{1}\setminus B_{\frac{1}{2}}$ and the barrier function $\Phi:\overline{\mathcal{A}_{\frac{1}{2}, 1}} \to \mathbb{R}_{+}$ given by
\begin{equation}\label{Phi}
    \Phi(x)=\mu_0 \cdot \left(e^{-\delta|x|^{2}} - e^{-\delta}\right)
\end{equation}
where $\mu_0, \delta>0$ will be chosen \textit{a posteriori}. Next, we observe that the gradient and the Hessian of $\Phi$ in $\mathcal{A}_{\frac{1}{2}, 1}$ are
$$    \nabla \Phi(x)=-2\mu_0\delta xe^{-\delta|x|^{2}} \quad \text{and} \quad   D^{2}\Phi(x)=2\mu_0\delta e^{-\delta |x|^{2}}\left(2\delta x\otimes x  - \text{Id}_{\mathrm{n}}\right).
$$

In the sequel, we will show that $\Phi$ is a strict viscosity subsolution to
\begin{equation}
     \mathcal{H}_{x_0, d_0}(x,\nabla \Phi)F_{x_0, d_0}(x, D^{2}\Phi) = f_{x_0, d_0}(x) \quad \text{in} \quad \mathcal{A}_{\frac{1}{2}, 1}
\end{equation}
provided we may adjust appropriately the values of $\mu_0, \delta>0$ and $r>0$.

First, notice that if we have $\delta > \frac{\Lambda(n-1)}{2\lambda}$, then $\Phi$ is a convex and decreasing function in the annular region $\mathcal{A}_{\frac{1}{2}, 1}$. This and the ellipticity of $F_{x_0, d_0}$ (see (A1)) give
\begin{align*}
F_{x_0, d_0}(x, D^{2}\Phi)& \geq \mathscr{P}^{-}_{\lambda,\Lambda}(D^2 \Phi(x)) \\
						  & = 2\mu_0\delta e^{-\delta|x|^{2}}\left[2\delta \lambda - \Lambda(n-1)\right] \\
						  & \geq 2\mu_0\delta e^{-\delta}\left[2\delta \lambda - \Lambda(n-1)\right] \quad \text{in} \quad \mathcal{A}_{\frac{1}{2}, 1}.
\end{align*}
Now \eqref{1.2} further gives
\begin{align*}
\mathcal{H}_{x_0, d_0}(x, \nabla\Phi) & =r^p\mathcal{H}\left(x_0+rd_0x, \frac{1}{r}\nabla\Phi\right)  \\
									  & \geq r^p\left(\frac{1}{r^p}|\nabla\Phi|^p+\mathfrak{a}(x_0+rd_0x)\frac{1}{r^q}|\nabla\Phi|^q\right) \\
									  &\geq (2\delta\mu_0e^{-\delta})^p\quad \text{in} \quad \mathcal{A}_{\frac{1}{2}, 1}
\end{align*}
(recall $q\geq p$). These two expressions together give
\begin{align*}
\mathcal{H}_{x_0, d_0}(x,\nabla \Phi)F_{x_0, d_0}(x, D^{2}\Phi) & \ge  (2\delta\mu_0e^{-\delta})^{p+1} \left[2\delta \lambda - \Lambda(n-1)\right]\\
   & >  r^{p+2}d_0\|f\|_{L^{\infty}\left(\mathcal{A}_{\frac{rd_0}{2}, rd_0}\right)},
\end{align*}
which holds true provided we choose $r\ll 1$ small (depending on $\mu_0$ and $\delta$). Therefore $\Phi$ is a strict subsolution.

Furthermore, we choose $\displaystyle \mu_0 \defeq (e^{-\delta/4}-e^{-\delta})^{-1}\cdot \inf_{\partial B_{\frac{1}{2}}} v_{x_0, d_0}(x)>0$. It follows that
$$
     \Phi(x) \leq v_{x_0, d_0}(x) \quad  \text{on} \quad  \partial \mathcal{A}_{\frac{1}{2}, 1}.
$$
Hence, from the Comparison Principle (see, Theorem \ref{comparison principle}), we can conclude that
 \begin{equation}\label{CompaBarV}
\Phi(x) \leq v_{x_0, d_0}(x) \quad  \text{in} \quad  \mathcal{A}_{\frac{1}{2}, 1}
 \end{equation}

Now, let $z_0\in \mathfrak{F}(v_{x_0,d_0})$ be a point that achieves the distance, i.e., $rd_0=|x_0-z_0|$ and consider $y_0\defeq \frac{z_0-x_0}{rd_0} \in \partial B_1$. Therefore, taking into account the free boundary condition, we obtain concerning the normal derivatives in the direction $\nu$ at $x_{0}$ the following

 \begin{equation}
     \mu_0\delta e^{-\delta} \leq \frac{\partial \Phi (y_{0})}{\partial \nu} \leq r\mathrm{Q}(y_0) \leq \|\mathrm{Q}\|_{L^{\infty}(B_1)}.
 \end{equation}
Therefore,
$$
\inf_{\partial B_{\frac{1}{2}}} v_{x_0, d_0}(x)\leq \|\mathrm{Q}\|_{L^{\infty}(B_1)}\delta^{-1}\cdot \left(e^{\frac{3}{4}\delta}-1\right) = \|\mathrm{Q}\|_{L^{\infty}(B_1)}\mathrm{C}(\delta).
$$

Now, by invoking the Harnack inequality (see, Theorem \ref{ThmHarIneq}) we conclude that
$$
\begin{array}{lcl}
  \displaystyle \sup_{B_{\frac{1}{2}}} v_{x_0, d_0}(x) & \leq & \displaystyle \mathrm{C}\cdot\left\{\inf_{B_{\frac{1}{2}}} v_{x_0, d_0}  + (q+1)^{\frac{1}{q+1}}\max\left\{(r^{p+2}d_0)^{\frac{1}{p+1}}, (r^{p+2}d_0)^{\frac{1}{q+1}}\right\}\Pi^{f, \mathfrak{a}_{x_0, d_0}}_{p, q}\right\} \\
   & \le &\displaystyle \mathrm{C}\cdot \left\{\inf_{\partial B_{\frac{1}{2}}} v_{x_0, d_0}  + (q+1)^{\frac{1}{q+1}}\max\left\{(r^{p+2}d_0)^{\frac{1}{p+1}}, (r^{p+2}d_0)^{\frac{1}{q+1}}\right\}\Pi^{f, \mathfrak{a}_{x_0, d_0}}_{p, q}\right\} \\
   & \le & \displaystyle \mathrm{C}\cdot \left\{\|\mathrm{Q}\|_{L^{\infty}(B_1)}\mathrm{C}(\delta) + (q+1)^{\frac{1}{q+1}}\max\left\{(r^{p+2}d_0)^{\frac{1}{p+1}}, (r^{p+2}d_0)^{\frac{1}{q+1}}\right\}\Pi^{f, \mathfrak{a}_{x_0, d_0}}_{p, q}\right\}.
\end{array}
$$
and from the definition of $v_{x_0, d_0}$, it follows that
$$
\displaystyle \sup_{B_{\frac{rd_0}{2}}(x_0)} u(x) \leq \mathrm{C}_0(\verb"universal")d_0\cdot\left\{\|\mathrm{Q}\|_{L^{\infty}(B_1)} + \max\left\{(r^{p+2}d_0)^{\frac{1}{p+1}}, (r^{p+2}d_0)^{\frac{1}{q+1}}\right\}\Pi^{f, \mathfrak{a}_{x_0, d_0}}_{p, q}\right\}
$$

Finally, by  $\mathcal{C}_{\text{loc}}^{1,\beta}$-estimates (see Theorem \ref{main1}) we have
	\begin{equation}\label{estimate gradient}
		|\nabla u(x_0)|=|\nabla v(0)|\leq \mathrm{C}\cdot\left(\|v_{x_0, d_0}\|_{L^{\infty}\left(B_{\frac{1}{2}}\right)}+1+\|f_{x_0, d_0}\|^{\frac{1}{p+1}}_{L^{\infty}\left(B_{\frac{1}{2}}\right)}\right).
	\end{equation}

\end{proof}

We now prove our non-degeneracy result. After constructing the appropriate barrier, the argument follows as the one in \cite{DeS11}, but we sketch it for completeness.

\begin{proof}[{\bf Proof of Theorem \ref{Nondeg}}]

Let $x_0\in B_{1/2}^+(u)$ and define as in the previous proof
	$$
	d_0\defeq \dist(x_0, \mathfrak{F}(u)).
	$$
We will suppose that $d_0 \leq \frac{1}{2}$ and consider again
 $$
    v_{x_0, d_0}(x) \defeq \frac{u(x_0+rd_0x)}{d_0}.
 $$
The aim is to show that
$$
v_{x_0, d_0}(0) \geq \mathrm{C}_{\ast}.
$$

Let $\Phi$ be defined as in \eqref{Phi} and
\[
\widehat{\Phi}(x) \defeq \mathrm{c}\cdot (1-\Phi(x))
\]
with $\mathrm{c}$ to be determined \textit{a posteriori}. On one hand, we can repeat the argument of the previous proof to get that $-\Phi$ is a strict viscosity supersolution to
\[
     \mathcal{H}_{x_0, d_0}(x,\nabla \Phi)F_{x_0, d_0}(x, D^{2}\Phi) = f_{x_0, d_0}(x) \quad \text{in} \quad \mathcal{A}_{\frac{1}{2}, 1}
\]
(and hence so is $\widehat{\Phi}$). Further, we choose $\mathrm{c}$ so that
\[
|\nabla\widehat{\Phi}|=\mathrm{c}|\nabla\Phi| <1-\eta_0
\]
on $\partial B_{1/2}$ so that $\widehat{\Phi}$ is a strict supersolution for the free boundary problem.

Assume without loss of generality that $\mathfrak{F}(u)$ is a graph in the $x_n$ direction with Lipschitz constant $L$ and, as a final piece of notation, let us denote
\[
\widehat{\Phi}_t(x) \defeq \mathrm{c}\cdot (1-\Phi(x+te_n)).
\]
Notice that for $t$ sufficiently large depending on $L$ we have that $\widehat{\Phi}_t$ lies above $v_{x_0, d_0}$ (which will be identically 0 eventually); we set $\hat{t}$ the smallest of such $t$. Next, we note that the touching point between $\widehat{\Phi}_{\hat{t}}$ and $v_{x_0, d_0}$ has to occur at some point $\hat{x}$ on the $\mathrm{c}$ level set
\[
v_{x_0, d_0}(\hat{x})=\mathrm{c}
\]
(since $v_{x_0, d_0}$ is a solution and $\widehat{\Phi}_{\hat{t}}$ is a supersolution) and $\dist(\hat{x},\mathfrak{F}(u))\leq \mathrm{L}$. Now,
\[
0<v_{x_0, d_0}(\hat{x})=\mathrm{c}\leq \mathrm{C}_1\text{dist}(\hat{x},\mathfrak{F}(u))
\]
where $\mathrm{C}_1$ is the constant from Theorem \ref{Lipschitz}, hence
\[
\frac{\mathrm{c}}{\mathrm{C}_1}\leq \text{dist}(\hat{x},\mathfrak{F}(u))\leq \mathrm{L}.
\]

This control above and below of $\dist(\hat{x},\mathfrak{F}(u))$ and the fact that $\mathfrak{F}(u)$ is Lipschitz allow us to connect $0$ and $\hat{x}$ with a (universal, depending on $\eta_0$ small) number of intersecting balls in which Harnack inequality applies and we get
\[
\frac{u(x_0)}{d_0}= v_{x_0, d_0}(0)\geq \mathrm{C}v_{x_0, d_0}(\hat{x})= \mathrm{C}\mathrm{c} \defeq \mathrm{C}_\ast,
\]
for a $\mathrm{C}>0$ coming from the Harnack inequality (Theorem \ref{ThmHarIneq}). This completes the proof.
\end{proof}

\section{A Harnack type inequality}\label{harnack}

In this section we will establish a Harnack type inequality (namely Theorem \ref{t 5.2.1}) for solutions to the free boundary problem \eqref{P 5.1. introduc} under the following smallness regime on the right hand side, the normal derivative and the oscillation of the coefficients:
\begin{eqnarray}
\label{t 5.2.b} \Vert  f \Vert_{L^{\infty}\left( \Omega \right)} \leq \varepsilon^{2}, \\
\label{t 5.2.d} \Vert \mathrm{Q} - 1 \Vert_{L^{\infty}\left(\Omega \right)} \leq \varepsilon^{2}, \\
\label{t 5.2.c} \Theta_F(x)\leq \varepsilon^{2},
\end{eqnarray}
for $0 < \varepsilon < 1$.

The proof relies on the following auxiliary Lemma:

\begin{lemma}\label{Harnack lemma 5.2}
There exists a constant $\tilde{\varepsilon}(\verb"universal") > 0$ such that if  $0 < \varepsilon \leq \tilde{\varepsilon}$, $u$ is a viscosity solution to (\ref{P 5.1. introduc}) in $\Omega$ and (\ref{t 5.2.b})--(\ref{t 5.2.c}) hold, then:
if
\begin{eqnarray}
\label{l 5.2.1}
\\
p^{+}\left( x \right) \leq u \left( x \right) \leq \nonumber \left( p\left( x \right) + \varepsilon \right)^{+}, \ \ \vert \sigma \vert < \frac{1}{20} \ \ \mbox{in} \ \ B_1,
\end{eqnarray}
with
\[
p\left( x \right) = x_n + \sigma
\]
and at $x_{0}= \frac{1}{10}e_n$
\begin{eqnarray}
\label{l 5.2.2}
u \left( x_{0} \right) \geq \left( p\left( x_{0} \right) + \frac{\varepsilon}{2} \right)^{+},
\end{eqnarray}
then
\begin{eqnarray}
\label{l 5.2.3}
u(x) \geq \left( p(x) + \mathrm{c}\varepsilon \right)^{+} \ \ \mbox{in} \ \ \overline{B}_{\frac{1}{2}},
\end{eqnarray}
for some $0 < \mathrm{c} < 1$. Analogously, if
\begin{eqnarray}
\label{l 5.2.4}
u \left( x_{0} \right) \leq \left( p\left( x_{0} \right) + \frac{\varepsilon}{2} \right)^{+},
\end{eqnarray}
then
\begin{eqnarray}
\label{l 5.2.5}
u(x) \leq \left( p(x) + \left( 1 - \mathrm{c} \right)\varepsilon \right)^{+} \ \ \mbox{in} \ \ \overline{B}_{\frac{1}{2}}.
\end{eqnarray}
\end{lemma}

\begin{proof}
We will certify \eqref{l 5.2.3}, the proof of \eqref{l 5.2.5} is similar. Furthermore, the degeneracy character of the operator naturally leads to split the proof into two steps, according to whether the gradient is ``large or small''. Before that, we recall that the interior estimates from \cite{daSR20} and \cite{DeF20} and the fact that
\[
B_{\frac{1}{20}}\left( x_{0} \right) \subset B^{+}_{1}\left( u \right).
\]
ensure, together with the ABP estimate from Theorem \ref{ABPthm}, that $u\in C^{1, \alpha}\left(B_{\frac{1}{40}}\left( x_{0} \right)\right)$ and
\[
\left[ u \right]_{1 + \alpha, B_{\frac{1}{40}}(x_{0})} \leq \mathrm{C} \left( \Vert u \Vert_{L^\infty(\Omega)} + \Vert f \Vert^{\frac{1}{p +1}}_{L^\infty(\Omega)} +1\right) \leq \mathrm{C}
\]
for a constant $\alpha(\verb"universal")  \in (0, 1)$ and $\mathrm{C}>1$ depending also on $\|f\|_{L^\infty(\Omega)}$ and diam($\Omega$).

Now we consider two cases:\\

{\bf Case 1:} If $\vert \nabla u (x_{0}) \vert < \frac{1}{4}$.\\

We can choose $r_{1} = r_{1}(p,q) > 0$ such that
\begin{eqnarray}
\label{NEW 1}
\vert \nabla u \vert \leq \dfrac{1}{2} \quad \text{in} \ B_{r_{1}}(x_{0}).
\end{eqnarray}
Further, if we take $r_{2}$ small enough depending on $r_{1}$ we can have
\begin{eqnarray*}
\label{NEW 2}
(x - r_{2} e_n) \in B_{r_{1}} (x_{0}), \quad \text{for all} \ x \in  B_{\frac{r_{1}}{2}}(x_{0}).
\end{eqnarray*}
Finally, let
\[
r_{3} = \min\left\lbrace \frac{r_{1}}{4}, \frac{r_{2}}{8} \right\rbrace.
\]
In $B_{2r_{3}}(x_{0})$ we can apply the Harnack inequality for nonhomogeneous operators (see Remark \ref{HarnIneqSV} in the Appendix) to get 	
\[
u\left( x \right) - p\left( x \right)  \geq  c_{0}(u\left( x_{0} \right) - p\left( x_{0} \right)) \geq \frac{c_{0}\varepsilon}{2}
\]
for all $x \in B_{r_{3}}(x_{0})$ or, denoting $\overline{x}_{0} = x_{0} - r_{2}e_n$,
\begin{eqnarray}
\label{NEW 5}
u (x) - p(x)\geq \frac{c_{0}\varepsilon}{2}
\end{eqnarray}
for all $x \in B_{r_{3}}(\overline{x}_{0})$.

Let now $w \defeq \overline{D} \rightarrow \mathbb{R}$ be defined by
\begin{eqnarray}
\label{l 5.2.6}
	w\left( x \right) = c\left( e^{-\delta| x - \overline{x}_{0}|}- e^{-\frac{4\delta}{5}}\right),
\end{eqnarray}
where $\mathrm{D} \defeq B_{\frac{4}{5}}\left( \overline{x}_{0} \right)\setminus \overline{B}_{r_{3}}\left( \overline{x}_{0} \right)$ and $\delta$ is a positive constant to be suitably chosen later and
\[
  \mathrm{c} \defeq \left( e^{-\delta r_3}- e^{-\frac{4\delta}{5}}\right)^{-1}
\]
so that
\begin{eqnarray}
\label{l 5.2.7}
w =	\left \{
\begin{array}{lll}
0 & \text{on} & \partial B_{\frac{4}{5}}\left( \overline{x}_{0} \right), \\
1 & \text{on} & \partial B_{r_{3}}\left( \overline{x}_{0}\right).
\end{array}
\right.
\end{eqnarray}
and we further extend $w\equiv1$ in $B_{r_{3}}\left( \overline{x}_{0} \right)$.

Next define
\begin{eqnarray}
\label{l 5.2.15}
v\left( x \right) = p\left( x \right) + \frac{c_{0}\varepsilon}{2} \left( w\left( x \right) - 1 \right), \ \ x \in \overline{B}_{\frac{4}{5}}\left( \overline{x}_{0} \right),
\end{eqnarray}
and for $t \geq 0$,
\begin{eqnarray}
\label{l 5.2.16}
v_{t}\left( x \right) = v\left( x \right) + t, \ \ x \in \overline{B}_{\frac{4}{5}}\left( \overline{x}_{0} \right)
\end{eqnarray}
and notice that
\begin{eqnarray}
\label{l 5.2.17}
v_{0}\left( x \right) = v\left( x \right) \leq p\left( x \right) \leq u \left( x \right), \ \ x \in \overline{B}_{\frac{4}{5}}\left( \overline{x}_{0} \right).
\end{eqnarray}
Consider
$$
	t_{0}= \sup \left\lbrace t \geq 0 : v_{t} \leq u \ \ \mbox{in} \ \ \overline{B}_{\frac{4}{5}}\left( \overline{x}_{0} \right) \right\rbrace.
$$

We claim that if we can show that $t_{0} \geq \frac{c_{0}\varepsilon}{2}$ we conclude the proof; indeed, in this scenario the definition of $v$ implies
\begin{eqnarray*}
\label{}
u\left( x \right) \geq v \left( x \right) + t_{0} \geq p\left( x \right) + \frac{c_{0}\varepsilon}{2} w\left( x \right), \ \ x \in B_{\frac{4}{5}}\left( \overline{x}_{0} \right).
\end{eqnarray*}
Also, notice that $\overline{B}_{\frac{1}{2}}\subset  B_{\frac{3}{5}}\left( \overline{x}_{0} \right)$ and $w\geq\tilde{c}>0$ in $B_{\frac{3}{5}}\left( \overline{x}_{0} \right)$. Hence, we conclude ($\varepsilon$ small) that
\[
u\left( x \right) - p\left( x \right) \geq \nonumber c_{1} \varepsilon \text{ in } B_{1/2},
\]
as desired.

The remains of the proof are therefore dedicated to show that indeed $t_{0} \geq \frac{c_{0}\varepsilon}{2}$. We suppose for the sake of contradiction that $t_{0} < \frac{c_{0}\varepsilon}{2}$. Then, there would exist $y_{0} \in \overline{B}_{\frac{4}{5}}\left( \overline{x}_{0} \right)$ such that
\begin{eqnarray*}
\label{}
v_{t_0}\left( y_{0} \right) = u\left( y_{0} \right).
\end{eqnarray*}
In the sequel, we show that $y_{0} \in B_{r_{3}}\left( \overline{x}_{0}\right)$. In fact, from the definition of $v_{t}$ and $w$ (that vanishes on $\partial B_{\frac{4}{5}}\left( \overline{x}_{0} \right)$) we have
\begin{eqnarray*}
\label{}
 	v_{t_0}(x)= p(x) - \frac{c_{0}\varepsilon}{2} + t_{0} < p(x)\leq u(x) \text{ on }\partial B_{\frac{4}{5}}\left( \overline{x}_{0} \right)
\end{eqnarray*}
so $y_0\notin\partial B_{\frac{4}{5}}\left( \overline{x}_{0} \right)$.

Let us show that $y_0$ cannot belong to $D$ either. We compute directly,
\begin{eqnarray}
\label{l 5.2.8}
\partial_{i}w = -2c\delta e^{-\delta|x-\overline{x}_0|}(x_i-\overline{x}_{0i})
\end{eqnarray}
where $\overline{x}_{0i}$ is the $i-$th component of $\overline{x}_{0}$ and
\begin{eqnarray}
\label{l 5.2.9}
\partial_{ij}w = \left \{
		\begin{array}{lll}
			4c\delta^2e^{-\delta|x-\overline{x}_0|}(x_i-\overline{x}_{0i})(x_j-\overline{x}_{0j}) \ \ \mbox{if} \ \ i\neq j \\
			-2c\delta e^{-\delta|x-\overline{x}_0|}\left(1-2\delta(x_i-\overline{x}_{0i})^2\right)\ \ \mbox{if} \ \ i=j. \\
		\end{array}
		\right.
\end{eqnarray}
Recall that \eqref{t 5.2.c} implies that $F(x,D^2w)$ is uniformly elliptic and note that if $\delta>\frac{\Lambda(n-1)}{2\lambda}$ we have
\[
F(x,D^2w)\geq \mathscr{P}^{-}_{\lambda,\Lambda}(D^2w(x))=2c\delta e^{-\delta|x-\overline{x}_0|}(2\delta\lambda-\Lambda(n-1))\geq 2c\delta e^{-\frac{4\delta}{5}}(2\delta\lambda-\Lambda(n-1)).
\]
Furthermore,
\[
\nabla v_{t_0}=e_{n} + \frac{c_{0} \varepsilon}{2} \nabla w \quad\text{and}\quad D^2v_{t_0}=\frac{c_{0} \varepsilon}{2}D^2w.
\]
Now
\[
\vert \nabla w(x) \vert \leq 2c\delta e^{-\delta r_3}\text{ in }D
\]
so for $ \varepsilon > 0$ small enough we have in $D$
\begin{eqnarray}
|\nabla v_{t_0}|\geq  \dfrac{1}{2} .
\end{eqnarray}
This and \eqref{1.2} imply
\[
\mathcal{H}(x, \nabla v_{t_0})\geq \eta
\]
for some $\eta>0$. This, the ellipticity of the operator and the previous computations give
\[
\mathcal{H}(x, \nabla v_{t_0})F(x,D^2v_{t_0}) \ge  \eta F\left(x,\frac{c_{0} \varepsilon}{2}D^2w\right)\geq \eta c_{0} \varepsilon c\delta e^{-\frac{4\delta}{5}}(2\delta\lambda-\Lambda(n-1)).
\]
Notice that $h(\delta) \defeq e^{-\frac{4\delta}{5}}\eta c_{0} c\delta (2\delta\lambda-\Lambda(n-1))$ satisfies
\[
h(\delta)>0\text{ for }\delta\in\left(\frac{\Lambda(n-1)}{2\lambda},\infty\right),
\]
and
\[
h\left(\frac{\Lambda(n-1)}{2\lambda}\right)=0 \quad \text{and} \quad \lim_{\delta\rightarrow\infty}h(\delta)=0.
\]
Then, choosing $\delta$ maximizing $h$ and for $\varepsilon$ small enough
\[
\mathcal{H}(x, \nabla v_{t_0})F(x,D^2v_{t_0}) \geq \varepsilon^2\geq f(x).
\]

On the other hand, recall that
\begin{eqnarray}
\label{t 5.2.18}
\vert \nabla v_{t_{0}} \vert \geq \vert \partial_nv \vert = \left\vert 1 + \frac{c_{0}\varepsilon}{2} \partial_n w \right\vert \ \  \mbox{in} \ \  D.
\end{eqnarray}
By radial symmetry of $w$, we have
\begin{eqnarray}
\label{t 5.2.19}
\partial_n w\left( x \right)  = \vert \nabla w\left( x \right) \vert \langle \nu_{x}, e_n\rangle , \ \ x \in D,
\end{eqnarray}
where $\nu_{x}$ is the unit vector in the direction of $x - \overline{x}_{0}$. From (\ref{l 5.2.8}) we have
\begin{equation}\label{t 5.2.20}
\vert \nabla w(x) \vert = 2c\delta e^{-\delta|x-\overline{x}_0|}|x_i-\overline{x}_{0i}|\geq 2c\delta e^{-\delta r_3}r_3>0.
\end{equation}
Also we have $\langle \nu_{x}, e_n\rangle \geq c$ in $\left\lbrace v_{t_{0}} \leq 0 \right\rbrace \cap D$ (for $\varepsilon$ small enough). In fact, if $\varepsilon$ is small enough
\begin{eqnarray*}
\label{}
\left\lbrace v_{t_{0}} \leq 0 \right\rbrace \cap D \subset \left\lbrace p \leq \frac{c_{0}\varepsilon}{2} \right\rbrace = \left\lbrace x_n \leq \frac{c_{0}\varepsilon}{2} - \sigma \right\rbrace \subset \left\lbrace x_n < \frac{1}{20} \right\rbrace.
\end{eqnarray*}
We therefore conclude that
\begin{eqnarray*}
\label{}
\langle \nu_{x}, e_{n}\rangle & = & \nonumber \frac{1}{\vert x-\overline{x}_{0} \vert}\langle x-\overline{x}_{0}, e_n \rangle \\ \nonumber
& \geq & \frac{5}{4}\langle x-\overline{x}_{0}, e_n \rangle  \\ \nonumber
& \geq & \frac{5}{4}\left(  \frac{1}{10} - r_{2} - x_n + \frac{1}{20} - \frac{1}{20}\right) \\ \nonumber
& > & c_{7}
\end{eqnarray*}
in $\left\lbrace v_{t_{0}} \leq 0 \right\rbrace \cap D$.

From this, (\ref{t 5.2.18}), (\ref{t 5.2.19}) and \eqref{t 5.2.20} we obtain
\begin{eqnarray*}
\vert \nabla v_{t_{0}} \vert^{2} & \geq &  \left(1 + \frac{c_{0}\varepsilon}{2}\vert \nabla w\left( x \right) \vert \langle \nu_{x}, e_n\rangle\right)^2 \\  \nonumber
& = & 1 + 2\tilde{c} \varepsilon^{}  + \tilde{c}  \varepsilon^{2}\vert \nabla w \vert^{2} \\  \nonumber
& \geq & 1 + 2c_{9}\varepsilon  + c_{10}\varepsilon^{2}  \\  \nonumber
&  \geq & 1 + \varepsilon^{2}
\end{eqnarray*}
and hence
\begin{eqnarray*}
\vert \nabla v_{t_{0}} \vert^{2}  \geq 1 + \varepsilon^{2} > \mathrm{Q}^2 \ \ \mbox{in} \ \ \left\lbrace v_{t_{0}} \leq 0 \right\rbrace \cap D.
\end{eqnarray*}
In particular, we have
\begin{eqnarray*}
\vert \nabla v_{t_{0}} \vert  > \mathrm{Q} \quad \mbox{in} \quad D \cap \mathfrak{F}\left( v_{t_{0}} \right).
\end{eqnarray*}
Thus, $v_{t_{0}}$ is a strict subsolution in $D$ and by Lemma \ref{l 5.1} we conclude that $y_{0}$ cannot belong to $D$ as desired.

In conclusion, $y_0 \in B_{r_{3}}\left( \overline{x}_{0} \right)$; but then
\begin{eqnarray*}
u\left( y_{0} \right) = v_{t_{0}}\left( y_{0} \right) = v\left( y_{0} \right) + t_{0} = p\left( y_{0} \right) + t_{0} < p\left( y_{0} \right) + \dfrac{c_{0}\varepsilon}{2}.
\end{eqnarray*}
which drives us to a contradiction to (\ref{NEW 5}) and this part of the proof is finished.

\bigskip

{\bf Case 2:} If $\vert \nabla u (x_{0}) \vert \geq \frac{1}{4}$.\\

Since $u\in C^{1, \alpha}\left(B_{\frac{1}{40}}(x_{0})\right)$, there exists a constant $r_{0} > 0$ such that
\begin{eqnarray}
\label{NEW 2.1}
\vert \nabla u \vert \geq \dfrac{1}{8} \quad \text{in} \ B_{r_{0}}(x_{0}).
\end{eqnarray}
which in particular implies (as in the previous step)
\[
\mathcal{H}(x,\nabla u )\geq c>0.
\]

Then, it is straightforward to see that $u$ satisfies
\begin{eqnarray}
\label{NEW 2.2}
F(x,D^2u) = \dfrac{f(x)}{ \mathcal{H}(x,\vert \nabla u \vert)}  \quad \text{in} \ B_{r_{0}}(x_{0}),
\end{eqnarray}
in the viscosity sense with
\[
\left|\dfrac{f(x)}{\mathcal{H}(x,\vert \nabla u \vert)}\right|\leq \frac{\varepsilon^2}{c}.
\]
Thus, by classical Harnack Inequality (found for instance in \cite{CC95}) we obtain
\begin{eqnarray*}
\label{}
u\left( x \right) - p\left( x \right)  & \geq & c_{0}(u\left( x_{0} \right) - p\left( x_{0} \right)) - C \Vert f \Vert_{\infty} \\
& \geq & \frac{c_{0}\varepsilon}{2} - C_{1}\varepsilon^{2} \\
& \geq & c_{1}\varepsilon,
\end{eqnarray*}
for all $x \in B_{1/40}(x_{0})$, if $\varepsilon >0$ is sufficiently small. The rest of the proof follows as in the previous case considering the same barrier $w$ defined in $B_{\frac{4}{5}}\left( x_{0} \right)\setminus B_{\frac{1}{40}}\left( x_{0} \right)$ instead of $B_{\frac{4}{5}}\left( x_{0} \right)\setminus B_{r_3}\left( x_{0} \right)$.
\end{proof}

\bigskip

Now, we establish the main result of this section, which is a straightforward consequence of the previous Lemma:

\begin{lemma}\label{t 5.2.1}
Let $u$ be a viscosity solution to (\ref{P 5.1. introduc}) in $\Omega$ under assumptions (\ref{t 5.2.b})--(\ref{t 5.2.d}). There exists a constant $\tilde{\varepsilon}(\verb"universal") > 0$ such that, if $u$ satisfies at some $x_{0} \in \Omega^{+}\left( u \right) \cup \mathfrak{F}\left( u \right)$,
\begin{eqnarray}
\label{t 5.2.2}
\left( x_n + a _{0}\right)^{+} \leq u \left( x \right) \leq \left( x_n + b_{0} \right)^{+} \ \ \mbox{in} \ \ B_{r}\left( x_{0} \right) \subset \Omega,
\end{eqnarray}
and
\begin{eqnarray}
\label{t 5.2.3}
b_{0} - a_{0} \leq \varepsilon r, \ \ \ \varepsilon \leq \tilde{\varepsilon}
\end{eqnarray}
then
\begin{eqnarray}
\label{t 5.2.4}
\left( x_n + a _{1}\right)^{+} \leq u \left( x \right) \leq \left( x_n + b_{1} \right)^{+} \ \ \mbox{in} \ \ B_{\frac{r}{40}}\left( x_{0} \right)
\end{eqnarray}
with
\begin{eqnarray}
\label{t 5.2.5}
a_{0} \leq a_{1} \leq b_{1} \leq b_{0}, \qquad b_{1} - a_{1} \leq \left( 1 - c \right)\varepsilon r,
\end{eqnarray}
and $0 < c(\verb"universal") < 1$.
\end{lemma}

\begin{proof}
With no loss of generality, we can assume  $x_{0}=0$ and $r=1$. Let us call $p\left( x \right) = x_n + a_{0}$ and notice that by (\ref{t 5.2.2})
\[
p^{+}\left( x \right) \leq u \left( x \right) \leq \left( p\left( x \right) + \varepsilon \right)^{+} \ \  \text{with } d_{0} = a_{0} + \varepsilon.
\]

Now, if $\left|a_0\right|<\frac{1}{20}$ we can apply the previous Lemma \ref{t 5.2.1} to get the desired result (either \eqref{l 5.2.2} or \eqref{l 5.2.4} must happen). The possibility $a_0<-\frac{1}{20}$ is a contradiction, since it implies that $0\in\{u=0\}$. Finally, if $a_0>\frac{1}{20}$ we are strictly contained in the positivity set of $u$ and the result follows by Harnack inequality.
\end{proof}

An immediate consequence of  Lemma \ref{t 5.2.1} is the following H\"older estimate:
\begin{corollary}
\label{c 5.1}
Let $u$ be a viscosity solution to (\ref{P 5.1. introduc}) in $\Omega$ under assumptions (\ref{t 5.2.b})--(\ref{t 5.2.d}). If $u$ satisfies (\ref{t 5.2.2}) then in $B_{1}\left( x_{0} \right)$ for $r=1$ then the function
\[
u_{\varepsilon}(x) \defeq \frac{u(x) - x_n}{\varepsilon}
\]
has a H\"older modulus of continuity at $x_{0}$ outside of ball of radius $\varepsilon / \tilde{\varepsilon}$; \textit{i.e.} for all $x \in \left( \Omega^{+}\left( u \right) \cup \mathfrak{F}\left( u \right) \right)\cap B_{1}\left( x_{0} \right)$ with $\vert x - x_{0} \vert \geq \varepsilon / \tilde{\varepsilon}$
\begin{eqnarray*}
\label{}
\vert \tilde{u}_{\varepsilon}\left( x \right) - \tilde{u}_{\varepsilon}\left( x_{0} \right)\vert \leq C \vert x - x_{0} \vert^{\gamma}.
\end{eqnarray*}
\end{corollary}

\begin{proof}
The proof holds by a standard iteration scheme. If $u$ satisfies \eqref{t 5.2.2} for $r=1$ then
\[
(x_n+a_1)^+\leq u(x)\leq (x_n+b_1)^+\text{ in } B_{\frac{r}{40}}(x_0)
\]
with
\[
b_1-a_1\leq (1-c)\varepsilon.
\]
Nevertheless, if
\[
(1-c)\varepsilon\leq 40^{-1}\tilde{\varepsilon}
\]
we can apply Lemma\ref{t 5.2.1} once more and get
\[
(x_n+a_2)^+\leq u(x)\leq (x_n+b_2)^+\text{ in } B_{\frac{r}{40^2}}(x_0)
\]
with
\[
b_2-a_2\leq (1-c)^2\varepsilon
\]
and we can repeat this argument as long as
\[
(1-c)^m\varepsilon\leq 40^{-m}\tilde{\varepsilon}.
\]
This means that the oscillation of $u_\varepsilon$ in $B_r(x_0)$ is smaller than $(1-c)^m=:40^{-\gamma m}$ as long as $r\geq \frac{\tilde{\varepsilon}}{\varepsilon}$ which yields the desired estimate.
\end{proof}

\section{Improvement of flatness scheme} \label{improvement}

In this section we prove an \textit{improvement of flatness} lemma, from which the proof of Theorem \ref{th 5.1. introd.} will follow via an iterative scheme.

\begin{lemma}[{\bf Improvement of flatness}]\label{lemma2N}
Let $u$ be a viscosity solution to (\ref{P 5.1. introduc}) in $\Omega$ under assumptions (\ref{t 5.2.b})--(\ref{t 5.2.d}) with $0 \in \mathfrak{F}(u)$ and assume it satisfies
\begin{equation}\label{(4.1N)}
    (x_n - \varepsilon)^{+} \le u(x) \le (x_n + \varepsilon)^{+} \quad \textrm{for}\,\, x \in B_1.
\end{equation}

Then there exists a constant $r_0(\verb"universal")>0$ such that if $0<r\leq r_0$ and $0 < \varepsilon \le \varepsilon_0$ (with $\varepsilon_0$ depending on $r$), then
\begin{equation}\label{(4.2)}
    \left(\langle x , \nu \rangle - r \frac{\varepsilon}{2}\right)^{+} \le u(x) \le \left(\langle x , \nu \rangle + r \frac{\varepsilon}{2}\right)^{+} \quad x \in B_r,
\end{equation}
for some $\nu\in\mathbb{S}^{n-1}$ (unity sphere) and $|\nu-e_n| \le C \varepsilon^2$ for a constant $C(\verb"universal")>0$.
\end{lemma}

\begin{proof}
We will split the proof into three steps. From now on, we will use the following notation:
$$
    \Omega_{\rho}(u) \defeq (B^{+}_{1}(u) \cup \mathfrak{F}(u)) \cap B_{\rho}.
$$

{\bf Step 1 - Compactness Lemma:}\, Fix $r \le r_0$ with $r_0(\verb"universal")$ ($r_0$ will be chosen in Step 3). Assume, for the sake of contradiction that we can find sequences $\varepsilon_{k} \to 0$ and  $\{u_{k}\}_{k \ge 1} \subset C(\Omega)$ viscosity solutions to
\begin{equation}\label{seq}
	\left \{
		\begin{array}{rclcl}
			\mathcal{H}(x, \nabla u_k) F_k(x, D^2 u_k) & = & f_k(x) & \text{ in } & \Omega_{+}\left( u_k \right),\\
			|\nabla u_k| & = & \mathrm{Q}_k(x) & \text{ on } & \mathfrak{F}(u_k),
		\end{array}
	\right.
\end{equation}
with
\[
\max\left\{\Vert  f_k \Vert_{L^{\infty}\left( \Omega \right)},\,\, \Vert \mathrm{Q}_k - 1 \Vert_{L^{\infty}\left(\Omega \right)}, \,\,\Theta_{\mathrm{F_k}}(x)\right\} \leq \varepsilon^{2}_k
\]
\begin{equation}\label{(4.3)}
    (x_n - \varepsilon_{k})^{+} \le u_{k}(x) \le (x_n + \varepsilon_{k})^{+} \quad \textrm{for} \,\, x \in B_{1}
\end{equation}
but it does not satisfy the conclusion \eqref{(4.2)} of Lemma.

Let $v_{k} : \Omega_{1}(u_k) \rightarrow \mathds{R}$ defined by
$$
    v_{k}(x) \defeq \frac{u_{k}(x)-x_n}{\varepsilon_{k}}.
$$
Then \eqref{(4.3)} gives,
\begin{equation}\label{(4.4)}
    -1 \le v_{k}(x) \le 1 \quad \textrm{for} \,\, x \in \Omega_{1}(u_k).
\end{equation}
From Corollary \ref{c 5.1}, it follows that the function $v_k$ satisfies
\begin{equation}\label{(4.5)}
    |v_{k}(x)-v_{k}(y)| \le \mathrm{C} |x-y|^{\gamma},
\end{equation}
for $\mathrm{C}(\verb"universal")$ and
$$
    |x-y| \ge \varepsilon_{k} / \bar \varepsilon, \,\,\, x,y \in \Omega_{1/2}(u_k).
$$
From \eqref{(4.3)} it follows that $\mathfrak{F}(u_k) \to B_1 \cap \{x_n=0\}$ in the Hausdorff distance (see Definition \ref{Def-Hausdorff-Dist}). This fact and \eqref{(4.5)} together with Arzel\`{a}-Ascoli give that as $\varepsilon_k \to 0$ the graphs of the $v_k$ over $\Omega_{1/2}(u_k)$ converge (up to a subsequence) in the Hausdorff distance to the graph of a H\"older continuous function $u_{\infty}$ over $B_{1/2} \cap \{x_n \ge 0\}$. \\

{\bf Step 2 - Limiting PDE:} We claim that $u_\infty$ is a solution of the problem
\begin{eqnarray}
\label{limiting theor}
\left \{
\begin{array}{rclcl}
\mathrm{F}_\infty(D^2u_{\infty}) & = 0 & \mbox{in} & B_{\frac{1}{2}}\cap \{x_n > 0\} \\
\frac{\partial u_{\infty}}{\partial x_n} & = 1 & \mbox{on} & B_{\frac{1}{2}} \cap \{x_n = 0\} \\
\end{array}
\right.
\end{eqnarray}
in viscosity sense, for some $\mathrm{F}_\infty$ uniformly elliptic with constant coefficients.

Notice that a stability argument and the uniform modulus of continuity of the operators $F_k$ gives the existence of such $\mathrm{F}_\infty$ and uniform convergence of $F_k$ (see for instance \cite[Lemma 2.2]{daSV21})
and the smallness assumption on $\Theta_{F_k}$ gives zero oscillation (for the coefficients) in the limit profile.

Now let $P \left( x \right)$ be a quadratic polynomial touching $u_\infty$ at $x_{0} \in B_{\frac{1}{2}}  \cap \left\lbrace  x_n \geq 0 \right\rbrace$ strictly by below. As shown in \cite{DeS11}, there exist points $x_k \in \Omega_{\frac{1}{2}}\left( u_k\right)$, $x_k \rightarrow x_{0}$, and constants $c_k \rightarrow 0$ such that
\begin{eqnarray*}
u_k \left( x_k\right) = \tilde{P}\left( x_k \right)
\end{eqnarray*}
and
\begin{eqnarray*}
u_k\left( x\right)\geq \tilde{P}\left( x \right) \ \ \mbox{in a neighborhood of} \ x_k
\end{eqnarray*}
where
\begin{eqnarray*}
\tilde{P}\left( x \right) = \varepsilon_{j}\left( P\left( x \right) + c_k  \right) + x_n.
\end{eqnarray*}

We need to prove that
\begin{enumerate}
	\item[(i)] If $x_{0} \in \ B_{\frac{1}{2}} \cap \{x_n > 0 \}$ then $\mathrm{F}_\infty(D^2 P) \leq 0$;
	\item[(ii)] If $x_{0} \in \ B_{\frac{1}{2}} \cap \left\lbrace  x_n = 0 \right\rbrace$ then $\frac{\partial P}{\partial x_n}\left( x_{0} \right) \leq 0$.
\end{enumerate}

\noindent (i) If $x_{0} \in B_{\frac{1}{2}} \cap \left\lbrace  x_n > 0 \right\rbrace$ then, since $P$ touches $u_k$ by below at $x_k$, we estimate
\[
\varepsilon_k^{2} \geq  f_k\left( x_k\right) \geq  \mathcal{H}(x_k, \nabla \tilde{P}) F_k(x_k, D^2 \tilde{P}).
\]

Now notice that
\[
\nabla \tilde{P} = \varepsilon_k \nabla P + e_n
\]
so that
\[
\mathcal{H}(x_k, \nabla \tilde{P})\geq \mathrm{c}
\]
for $k$ large enough ($\mathrm{c}\in (0,1)$) so we may take limit in
\[
\frac{\varepsilon_k^{2}}{\mathcal{H}(x, \nabla \tilde{P})} \geq  F_k(x_k, D^2 \tilde{P})
\]
to get
\[
0 \geq  \mathrm{F}_\infty(D^2 \tilde{P}).
\]

\noindent (ii)  If $x_0\in \ B_{\frac{1}{2}} \cap \left\lbrace  x_n = 0 \right\rbrace$ we can assume that
\begin{eqnarray}
\label{strict subharmonic}
\mathrm{F}_\infty(D^2P) > 0
\end{eqnarray}
Notice that for $k$ sufficiently large we have $x_k\in \mathfrak{F}\left( u_k\right)$. In fact, suppose by contradiction that there exists a subsequence $x_{k_j} \in B^{+}_{1}\left( u_{k_j}\right)$ such that $x_{k_j} \rightarrow x_{0}$. Then arguing as in (i) we obtain
\begin{eqnarray*}
F_k(x_{k_j},D^2P) \leq  C\varepsilon_k
\end{eqnarray*}
whose limit implies
\begin{eqnarray*}
\mathrm{F}_\infty(D^2P) \leq 0,
\end{eqnarray*}
which contradicts (\ref{strict subharmonic}).
Therefore, there exists $k_{0} \in \mathbb{N}$ such that $x_k \in \mathfrak{F}\left( u_k\right)$ for $k \geq k_{0}$.

Moreover, as in the previous step
\begin{eqnarray*}
\vert \nabla \tilde{P} \vert \geq \frac{1}{2},
\end{eqnarray*}
for $k$ sufficiently large. Since that $\tilde{P}^{+}$ touches $u_k$ by below we have
\begin{eqnarray*}
\vert \nabla \tilde{P} \vert ^{2} \leq \mathrm{Q}_k\left( x_k\right) \leq \left(1 + \varepsilon^{2}_k \right)   .
\end{eqnarray*}
Then, we obtain
\begin{eqnarray*}
\vert \nabla \tilde{P} \vert ^{2} \leq \left(1 + \varepsilon^{2}_k \right).
\end{eqnarray*}
Moreover,
\begin{eqnarray*}
\vert \nabla \tilde{P} \vert ^{2} & = &  \varepsilon^{2}_k\vert \nabla P \left( x_k\right) \vert^{2} + 1 + 2\varepsilon_k \frac{\partial P}{\partial x_n} \left( x_k\right).
\end{eqnarray*}
Putting the last to equations together and dividing by $\varepsilon_k$ we get
\begin{eqnarray}
\label{t 5.3.4}
\varepsilon_k\vert \nabla P \left( x_k\right) \vert^{2} + 2 \frac{\partial P}{\partial x_n} \left( x_k\right) \leq \varepsilon_k
\end{eqnarray}
and taking $j \rightarrow \infty$ we conclude that $\frac{\partial P}{\partial x_n} \left( x_0\right) \leq 0$.   \\

{\bf Step 3 - Improvement of flatness:}\, So far we have that $u_{\infty}$ solves \eqref{limiting theor} and from \eqref{(4.4)} it satisfies
$$
    -1 \le u_{\infty} \le 1 \quad \textrm{in} \,\, B_{1/2} \cap \{x_n \ge 0\}.
$$
From Lemma \ref{reg_Finfty} and the bound above we obtain that, for the given $r$,
$$
|u_{\infty}(x) - u_{\infty}(0) - \langle \nabla u_{\infty}(0) , x \rangle| \le C_0 r^{1+\alpha} \quad \textrm{in} \,\, B_r \cap \{x_n \ge 0\},
$$
for a constant $C_0(\verb"universal")$. In particular, since $0 \in \mathfrak{F}(u_{\infty})$ and $\partial_n u_{\infty}(0)=0$, we estimate
$$
   \langle \tilde{x} , \tilde{\nu} \rangle - C_1 r^{1+\alpha} \le u_{\infty}(x) \le \langle \tilde{x} , \tilde{\nu}\rangle + C_0 r^{1+\alpha} \quad \textrm{in} \,\, B_r \cap \{x_n \ge 0\},
$$
where $\tilde{\nu}_{i} = \langle \nabla u_{\infty}(0) , e_i \rangle$ $(i=1,\ldots, n-1)$, $|\tilde{\nu}| \le \tilde{C}$ and $\tilde{C}(\verb"universal")$. Therefore, for $k$ large enough we get,
$$
    \langle \tilde{x} , \tilde{\nu} \rangle - C_1 r^{1+\alpha} \le v_{k}(x) \le \langle \tilde{x} , \tilde{\nu} \rangle + C_1 r^{1+\alpha} \quad \textrm{in} \,\, \Omega_{r}(u_k).
$$
From the definition of $v_{k}$ the inequality above reads
\begin{equation}\label{(4.9)}
\varepsilon_k\langle \tilde{x} , \tilde{\nu} \rangle +x_n- C_1 \varepsilon_kr^{1+\alpha} \le v_{k}(x) \le \langle \tilde{x} , \tilde{\nu} \rangle +x_n+ \varepsilon_kC_1 r^{1+\alpha} \quad \textrm{in} \,\, \Omega_{r}(u_k).
\end{equation}

Now, let us define
$$
    \nu \defeq \frac{1}{\sqrt{1+ \varepsilon^{2}_{k}}}(\varepsilon_{k} \tilde{\nu}, 1).
$$
Since, for $k$ large,
$$
 1 \le \sqrt{1+\varepsilon^{2}_{k}} \le 1+ \frac{\varepsilon^{2}_{k}}{2},
$$
we conclude from \eqref{(4.9)} that
$$
\langle x , \nu \rangle - \frac{\varepsilon^{2}_{k}}{2} r - C_1 r^{1+\alpha} \varepsilon_{k} \le u_{k} \le \langle x , \nu \rangle + \frac{\varepsilon^{2}_{k}}{2} r + C_1 r^{1+\alpha} \varepsilon_k \qquad \textrm{in}\qquad, \Omega_{r}(u_k).
$$
In particular, if $r_0$ is such that $C_1r_0^\alpha \le \frac{1}{4}$ and also $k$ is large enough so that $\varepsilon_{k} \le \frac{1}{2}$ we find
$$
    \langle x , \nu \rangle - \frac{\varepsilon_{k}}{2} r \le u_k \le \langle x , \nu \rangle + \frac{\varepsilon_{k}}{2} r \qquad \textrm{in} \qquad \Omega_{r}(u_k),
$$
which together with \eqref{(4.3)} implies that
$$
    \left(\langle x , \nu \rangle - \frac{\varepsilon_{k}}{2} r\right)^{+} \le u_k \le  \left(\langle x , \nu \rangle + \frac{\varepsilon_{k}}{2} r\right)^{+} \quad \textrm{in}\,\,\, B_r.
$$
Finally, such a $u_k$ satisfies the conclusion of Lemma, thereby yielding a contradiction.

\end{proof}

%%%%%%%%%%%%%%%%%%%%%%%%%%%%%%%%%%%%%%%%%%%%%%%%%%%%%%%%%%%%%%%%%%%%%%%%%
%%%%%%%%%%%%%%%%%%%%%%%%%%%%%%%%%%%%%%%%%%%%%%%%%%%%%%%%%%%%%%%%%%%%%%%%%

\section{Regularity of the free boundary}\label{reg_fron_livre}

In this section we will present the proof of Theorem \ref{th 5.1. introd.}. Subsequently, via a blow-up argument performed in such a result we will deliver the proof of Theorem \ref{Holder1}. The proof of Theorem \ref{th 5.1. introd.} is based on an improvement of flatness coming from Harnack type estimates and it follows closely the ideas of \cite{DeS11}.

%Hereafter, we will assume
%\begin{equation} \label{H}
%	|Q(x)-Q(y)| \le \tau(|x-y |) \quad \textrm{for} \,\,\, x,y \in B_1,
%\end{equation}
%where $\tau$ is any modulus of continuity (i.e. a continuous increasing function such that $\tau(0)=0$).

%%%%%%%%%%%%%%%%%%%%%%%%%%%%%%%%%%%%%%%%%%%%%
%%%%%%%%%%%%%%%%%%%%%%%%%%%%%%%%%%%%%%%%%%%%%

\begin{proof}[{\bf Proof of Theorem \ref{th 5.1. introd.}}]

The idea of proof is to iterate the Lemma \ref{lemma2N} in an appropriate geometric scaling. For that end, let $u$ be a viscosity solution to the free boundary problem
\begin{eqnarray*}
	\left \{
		\begin{array}{rclcl}
			\mathcal{H}(x, \nabla u) F (x,  D^2 u) &   = & f(x) & \text{ in } & B^+_1(u),\\
			|\nabla u| & = & \mathrm{Q}(x) & \text{ on } & \mathfrak{F}(u)
		\end{array}
	\right.
\end{eqnarray*}
with $0 \in \mathfrak{F}\left( u \right)$ and $\mathrm{Q}\left(0 \right)= 1$. Now, assume further that
\[
\left(x_{n} - \tilde{\varepsilon}\right)^{+} \leq u \left( x \right) \leq \left( x_{n} + \tilde{\varepsilon} \right)^{+} \ \ \mbox{for} \ \ x \in B_{1},
\]
and
\[
\max\left\{\Vert f \Vert_{L^{\infty}\left( B_{1} \right)}, \ \ \left[ \mathrm{Q} \right]_{C^{0,\alpha}\left( B_{1}\right)}, \|F\|_{C^\omega(B_1)}\right\} \leq \tilde{\varepsilon},
\]
with $\tilde{\varepsilon}>0$ to be fixed soon.

Let us start by fixing $\bar r(\verb"universal") >0$ to be a constant such that
\begin{equation} \label{smallRN}
	\overline{r} \le \min\left\{r_0,\,\,\left(\frac{1}{4}\right)^{\frac{1}{\alpha}}\right\},
\end{equation}
with $r_0(\verb"universal")$ the constant in Lemma \ref{lemma2N}. After chosen $\overline{r}$, let $\varepsilon_{0} \defeq \varepsilon_{0}(\overline{r})$ be the constant given by Lemma \ref{lemma2N}.

Now, let
\begin{equation}\label{choiceN}
	\tilde{\varepsilon} \defeq \varepsilon^{2}_{0} \quad \textrm{and} \quad \varepsilon_{k} \defeq 2^{-k} \varepsilon_0.
\end{equation}
Notice that our choice of $\tilde{\varepsilon}$ ensures that
\[
	\left(x_n -  \varepsilon_0\right)^{+} \le u(x) \le \left(x_n +  \varepsilon_0\right)^{+} \quad \textrm{in} \quad B_1.
\]
Thus by Lemma \ref{lemma2N} there exists $\nu_1$ with $|\nu_{1}|=1$ and $|\nu_{1}-e_n| \le \mathrm{C}{ \varepsilon_0}^{2}$ such that
$$
	\left(\langle x , \nu_{1} \rangle - \bar r \frac{\varepsilon_0}{2}\right)^{+} \le u(x) \le \left(\langle x , \nu_{1}\rangle + \bar r \frac{\varepsilon_0}{2}\right)^{+} \quad \textrm{in} \quad B_{\bar r}.
$$

{\bf Smallness regime:}  Consider the sequence of re-scaling profiles $u_{k}: B_1 \rightarrow \mathbb{R}$ given by
$$
    u_{k}(x) \defeq \frac{u(\lambda_{k}x)}{\lambda_{k}}
$$
with $\lambda_{k} = \overline{r}^{k}$, $k=0,1,2,\ldots$, for a fixed $\overline{r}$  as in \eqref{smallRN}.
Then, we observe that $u_{k}$ fulfils in the viscosity sense the following free boundary problem
\begin{eqnarray}
 \label{}
	\left \{
		\begin{array}{rclcl}
			\mathcal{H}(\lambda_k x, \nabla u_k) F_k( x,  D^2 u_k) & = & f_{k}(x) & \text{ in } & B^+_1(u_{k}),\\
			|\nabla u_{k}| & = & \mathrm{Q}_{k} & \text{ on } & \mathfrak{F}(u_{k}),
		\end{array}
	\right.
\end{eqnarray}
where
$$
\left\{
\begin{array}{rcl}
  F_{x}(x, \mathrm{X}) & \defeq & \lambda_k F\left(\lambda_k x, \lambda^{-1}_k\mathrm{X}\right) \\
  \mathcal{H}_{k}(x, \xi) & \defeq &  \mathcal{H}\left(\lambda_kx, \xi\right)\\
  \mathfrak{a}_{k}(x) & \defeq & \mathfrak{a}(\lambda_k x)\\
  f_{k}(x) & \defeq & \lambda_kf(\lambda_k x)\\
  \mathrm{Q}_{k}(x)&\defeq& \mathrm{Q}(\lambda_k x).
\end{array}
\right.
$$
Furthermore, $F_{k}, \mathcal{H}_{k}$ and $\mathfrak{a}_{k}$ fulfil the structural assumptions (A0)-(A2), \eqref{1.2} and \eqref{1.3}.

Now, we also claim that for the choices made in \eqref{choiceN} the assumptions (\ref{t 5.2.b})--(\ref{t 5.2.d}) hold true. Indeed, in $B_1$ we have
\begin{eqnarray*}
  |f_{k}(x)| &\le& \lambda_k\|f\|_{L^{\infty}(B_1)}  \le \tilde{\varepsilon} {\bar r}^k  = \varepsilon_k^{2} (4{\bar r})^k \leq  \varepsilon_{k}^{2},\\
|Q_k(x) -1| &=& |Q(\lambda_k x) - Q_k(0)| \le [Q]_{C^\alpha(B_1)} \lambda^{\alpha}_{k}  \le \tilde{\varepsilon} {\bar r}^ {k \alpha}  \le (\varepsilon_0 2^{-k})^{2} = \varepsilon_{k}^{2}\\
\Theta_{F_k}(x) &\le& \varepsilon_{k}^{2}.
\end{eqnarray*}

Therefore, we can iterate the above argument and obtain that
\[
    ( \langle x , \nu_{k} \rangle -\varepsilon_{k})^{+} \le u_{k}(x) \le (\langle x , \nu_{k} \rangle + \varepsilon_{k})^{+} \quad \textrm{in} \quad B_1,
\]
with $|\nu_{k}|=1$, $|\nu_{k}-\nu_{k+1}| \le C \varepsilon_{k}$ (with $\nu_{0}=e_n$), with $\mathrm{C}(\verb"universal")$. Thus, we have
\begin{equation} \label{int2}
	\left( \langle x , \nu_k \rangle - \frac{\varepsilon_0}{2^k} \overline{r}^{k}\right)^{+} \le u(x) \le \left(\langle  x , \nu_k \rangle  + \frac{\varepsilon_{0}}{2^{k}} \overline{r}^{k}\right)^{+} \quad \textrm{in} \quad B_{\overline{r}^{k}}
\end{equation}
with
\begin{equation} \label{int3}
	|\nu_{k+1} - \nu_{k}| \le \mathrm{C} \frac{\varepsilon_0}{2^k}.
\end{equation}
Furthermore, \eqref{int2} implies that
\begin{equation} \label{int5}
	\partial \{u>0\} \cap B_{\overline{r}^k} \subset \left\{|\langle x, \nu_k \rangle| \le \frac{\varepsilon_0}{2^k} \overline{r}^{k}\right\},
\end{equation}
which implies that $B_{3/4} \cap \mathfrak{F}(u) $ is a $C^{1,\beta}$ graph. In fact, by \eqref{int3} we have that $\{\nu_k\}_{k \ge 1}$ is a Cauchy sequence. Hence, there exists the limit
$$
	\nu_\infty \defeq \lim_{k \to \infty} \nu_k
$$
In addition, from \eqref{int3} we conclude
$$
	|\nu_k - \nu_\infty| \le \mathrm{C} \frac{\varepsilon_0}{2^k} \quad \text{for}\,\,\, k\,\,\,\text{large enough}.
$$

Now, from \eqref{int5} we have
\[
	|\langle x , \nu_{k}\rangle| \le \frac{\varepsilon_0}{2^k} \overline{r}^k.
\]
Next, fix $x \in B_{3/4} \cap \partial \{u >0\}$ and choose $k$ an integer such that
$$
	\overline{r}^{k+1} \le |x| \le \overline{r}^k.
$$
Then,
\begin{eqnarray*}
	|\langle x , \nu_\infty \rangle| &\le& |\langle x, \nu_\infty- \nu_k \rangle| + |\langle x , \nu_k\rangle| \\
	&\le& |\nu(0) - \nu_k| |x| + \frac{\varepsilon_0}{2^k} \overline{r}^k\\
	&\le& C\frac{\varepsilon_0}{2^k} |x| + \frac{\varepsilon_0}{2^k} \overline{r}^k\\
	&\le& C\frac{\varepsilon_0}{2^k} (|x| + \overline{r}^k) \\
	&=& C \frac{\varepsilon_0}{2^k} \left(|x| + \frac{\overline{r}^{k+1}}{\overline{r}}\right)\\
	&\le& C \frac{\varepsilon_0}{2^k} \left(1 + \frac{1}{\overline{r}}\right) |x|.
	\end{eqnarray*}
	From the choice of $k$, we have  $|x| \ge \overline{r}^{k+1}$. Hence, if we define $0 < \beta <1$ such that
	$$
		\frac{1}{2}^{} = \overline{r}^{\beta} \quad \Leftrightarrow \quad \beta \defeq \frac{\ln(2)}{\ln(\overline{r}^{-1})},
	$$
we have
	\begin{eqnarray*}
		|\langle x , \nu_\infty\rangle| &\le& C \left(\frac{1}{2}\right)^k \left(1 + \overline{r}^{-1}\right) |x| \\
		&=& C \left(\frac{1}{2}\right)^{k+1} \left(1 + \overline{r}^{-1}\right) 2 |x| \\
		&\le& C (1+\overline{r}^{-1}) \varepsilon_0 |x|^{1+\beta}\\
        &\le& C \varepsilon_0 |x|^{1+\beta}.
	\end{eqnarray*}
Finally, we obtain
$$
	\partial\{u >0\} \cap B_{\overline{r}^{k}} \subset \left\{ \langle x, \nu_\infty \rangle \le C \varepsilon_0 \overline{r}^{k(1+ \beta)}\right\},
$$
which implies that $\partial \{u>0\}$ is a differentiable surface at $0$ with normal $\nu_\infty$. By applying this argument at all points in $\partial\{u>0\} \cap B_{3/4}$ we see that $\partial \{u>0\} \cap B_{3/4}$ is  a $C^{1,\beta}$ surface.
\end{proof}

\begin{remark}
By appropriately modifying the proof above, the same result could be proved if the free boundary condition is assumed to satisfy the more general continuity assumption
\[
|\mathrm{Q}(x)-\mathrm{Q}(y)|\leq \omega^{\prime}(|x-y|)
\]
for some modulus of continuity which is not necessarily a power.
\end{remark}

%%%%%%%%%%%%%%%%%%%%%%%%%%%%%%%%%%%%%%%%%%%%%%
%%%%%%%%%%%%%%%%%%%%%%%%%%%%%%%%%%%%%%%%%%%%%%

The next point to address is the proof of Theorem \ref{Holder1}. We present first two preliminary Lemmas. The first lemma is standard and follows essentially by a similar result as in \cite{DeS11}.

\begin{lemma}[{\bf Compactness}] \label{compactness}
Let $u_k$ be a sequence of (Lipschitz) viscosity solutions to
$$
\left\{
\begin{array}{rclcl}
\mathcal{H}(x, \nabla u_k) F_k(x,D^2 u_k)&=& f_k(x) & \mbox{in} & \Omega^{+}(u_k) ,\\
|\nabla u_k | &=&\mathrm{Q}_k(x) & \mbox{on} & \mathfrak{F}(u_k).
\end{array}
\right.
$$
Assume further that
\begin{enumerate}
\item[(i)]
$
    u_{k} \to u_{\infty} \quad \text{uniformly on compacts};
$
\item[(ii)] $F_k \to \mathrm{F}_{\infty}$ locally uniformly on $\text{Sym}(n) \times \mathbb{R}^n$;
\item[(iii)]
$
    \partial \{u_{k} >0\} \to \partial \{u_{\infty} >0\} \quad \textrm{locally in the Hausdorff distance;}
$
\item[(iv)] $\max\left\{\|f_k\|_{L^{\infty}}, \,\,\|Q_k-1\|_{L^{\infty}}, \,\,\Theta_{F_k}\right\}= o(1)$, as $k \to \infty$.
\end{enumerate}
Then, $u_{\infty}$ be a viscosity solution of
$$
\left\{
\begin{array}{rclcl}
\mathrm{F}_{\infty}(D^2 u_{\infty} ) &=& 0&  \mbox{in} & \Omega^{+}(u_{\infty}) ,\\
\frac{\partial u_{\infty}}{\partial x_n} &=& 1 & \mbox{on} & \mathfrak{F}(u_{\infty}),
\end{array}
\right.
$$
\end{lemma}

\begin{proof}
The proof is rather similar to the one of Step 2 in Lemma \ref{lemma2N} so we skip it.
\end{proof}

Finally, we will use the following Liouville type result for  global viscosity solutions to a one-phase homogeneous free boundary problem; the result is more general and is proved in \cite{DFS15}, but we restate it here in way suitable for our context:
\begin{lemma}\label{Liouville}
Let $v: \mathbb{R}^{n} \rightarrow \mathbb{R}$ be a non-negative viscosity solution to
$$
\left\{
\begin{array}{rcrcl}
 \mathrm{F}(D^2 v) &=& 0 &\mbox{in} & \{v>0\}\\
\frac{\partial v}{\partial \nu} & = & 1 & \mbox{on} & \mathfrak{F}(v)
\end{array}
\right.
$$
 Assume further that $\mathfrak{F}(v) = \left\{x_n =g(x^{\prime}) \,\,\,\text{for}\,\,\, x^{\prime} \in \mathbb{R}^{n-1}\right\}$ with $\mathrm{Lip}(g) \le \mathrm{M}$. Then, $g$ is linear function and
$$
    v(x)= (x\cdot e)^{+},
$$
for some unit vector $e$, i.e. $v$ is a one dimensional linear profile.
\end{lemma}

%%%%%%%%%%%%%%%%%%%%%%%%%%%%%%%%%%%%%%%%%%%%%%%%%
%%%%%%%%%%%%%%%%%%%%%%%%%%%%%%%%%%%%%%%%%%%%%%%%%

Finally, we are in a position of presenting the proof of Theorem \ref{Holder1}.

\begin{proof}[{\bf Proof of Theorem \ref{Holder1}}]
Let $\tilde{\varepsilon}(\verb"universal")>0$ be the constant from Theorem  \ref{th 5.1. introd.}. Without loss of generality, we may assume $\mathrm{Q}(0)=1$. Now, consider the re-scaled functions
$$
    u_k(x) \defeq u_{\delta_{k}}(x) = \frac{u(\delta_{k}x)}{\delta_k},
$$
with $\delta_{k} \to 0$ as $k \to \infty$. Notice that $u_k$ solves in the viscosity sense
$$
\left\{
\begin{array}{rclcl}
\mathcal{H}_k(x, \nabla u_k)F_k(x,D^2u_k)&=& f_k(x) & \mbox{in} & B^{+}_{1}(u_k),\\
| \nabla u_{k} |&=& \mathrm{Q}_k(x) & \mbox{on} & \mathfrak{F}(u_k),
\end{array}
\right.
$$
where,
$$
\left\{
\begin{array}{rcl}
  F_{x}(x, \mathrm{X}) & \defeq & \delta_k F\left(\delta_k x, \delta_k^{-1}\mathrm{X}\right) \\
  \mathcal{H}_{k}(x, \xi) & \defeq &  \mathcal{H}\left(\delta_k x, \xi\right)\\
  \mathfrak{a}_{k}(x) & \defeq & \mathfrak{a}(\delta_k x)\\
  f_{k}(x) & \defeq & \delta_kf(\delta_k x)\\
  \mathrm{Q}_{k}(x)&\defeq& \mathrm{Q}(\delta_k x).
\end{array}
\right.
$$
with $F_{k}, \mathcal{H}_{k}$ and $\mathfrak{a}_{k}$ fulfilling the structural assumptions (A0)-(A2), \eqref{1.2} and \eqref{1.3}.
Furthermore, for $k$ large, the smallness conditions are satisfied for a constant $\tilde\varepsilon(\verb"universal")$. In fact, in $B_1$ we have
\begin{eqnarray*}
   |f_{k}(x)| &=& \delta_{k} |f(\delta_{k}x)| \le \delta_{k} \|f\|_{L^{\infty}} \le \tilde{\varepsilon}^{2}\\
    |\mathrm{Q}_{k}(x)-1| &=& |\mathrm{Q}_{k}(x)-\mathrm{Q}_{k}(0)| \le  \tau(1) \delta^{\beta}_{k} \le \tilde{\varepsilon}^{2}\\
    \Theta_{F_k}(x) &\le& \tilde{\varepsilon}^{2}
\end{eqnarray*}
for $k$ large enough. Therefore, using non-degeneracy and uniform Lipschitz continuity of the $u_{k}'s $ provided by Theorems \ref{Lipschitz} and \ref{Nondeg}, standard arguments imply that, up to a subsequence,
\begin{enumerate}
\item[(i)] There exists $u_{\infty} \in C(\Omega)$ such that
$
    u_{k} \to u_{\infty} \quad \textrm{uniformly on compact sets};
$
\item[(ii)] There exists $\mathrm{F}_{\infty}: \text{Sym}(n) \rightarrow \mathbb{R}$ such that $F_k \to \mathrm{F}_{\infty}$ locally uniformly ;
\item[(iii)]
$
    \partial \{u_{k} >0\} \to \partial \{u_{\infty} >0\} \quad \textrm{locally in the Hausdorff distance;}
$
\item[(iv)] $\max\left\{\|f_k\|^{\frac{1}{p+1}}_{L^{\infty}(B_1)}, \|\mathrm{Q}_k-1\|_{L^{\infty}(B_1)}, \Theta_{F_k}(x) \right\} = \text{o}(1)$ as $k \to \infty$.
\end{enumerate}
Now, as in Lemma \ref{compactness} and using a Cutting Lemma as one in \cite[Lemma 6]{IS12},
the blow-up limit $u_{\infty}$ solves the global homogeneous one-phase free boundary problem
 $$
\left\{
\begin{array}{rclcl}
\mathrm{F}_{\infty}(D^2 u_{\infty}) &=& 0& \mbox{in} &  \{u_{\infty}>0\},\\
|\nabla u_{\infty}|&=&1   & \mbox{on} & \mathfrak{F}(u_{\infty}).
\end{array}
\right.
$$
Furthermore, since $\mathfrak{F}(u_k)$ is a Lipschitz graph in a neighborhood of $0$, we also have from items $(i)-(iii)$ that $\mathfrak{F}(u_{\infty})$ is Lipschitz continuous. Thus, from Lemma \ref{Liouville}  we conclude that $u_{\infty}$ is a one-phase solution, i.e. up to rotations,
$$
    u_{\infty}(x) = x^{+}_{n}.
$$
Thus, for $k$ large enough we have
$$
    \|u_{k}-u_{\infty}\|_{L^{\infty}(B_1)} \le \tilde{\varepsilon}.
$$

By combining the above facts, one concludes that for all $k$ large enough, $u_k$ fulfils $\tilde{\varepsilon}$-flat in $B_1$ (see Theorem \ref{th 5.1. introd.}). Thus,
$$
    \left(x_n-\tilde{\varepsilon}\right)^{+} \le u_{k}(x) \le \left(x_n+\tilde{\varepsilon}\right)^{+}, \,\,\, x \in B_1.
$$
Therefore, $\mathfrak{F}(u_k)$ is a graph $C^{1,\gamma}$  and consequently $\mathfrak{F}(u)$ are $C^{1,\gamma}$, for some $\gamma \in (0,1)$. This completes the proof.
\end{proof}

\appendix
\section{Appendix}

\subsection{A Harnack inequality for doubly degenerate elliptic PDEs}

For the reader's convenience, in what follows we gather the statements of two fundamental results in elliptic regularity, namely the Weak Harnack inequality and the Local Maximum Principle. Such pivotal tools will provide a Harnack inequality (resp. local H\"{o}lder regularity) to viscosity solutions.

\begin{theorem}[{\bf Weak Harnack inequality, \cite[Theorem 2]{I11}}]\label{wharn}
Let $u$ be a non-negative continuous function such that
$$
   F_0(x, \nabla u, D^2u)\leq 0 \quad\textrm{ in }\quad B_1
$$
in the viscosity sense. Assume that $F_0$ is uniformly elliptic in the $X$ variable (see (A1) condition) and $F_0 \in C^0(B_1\times \left(\R^n \setminus B_{\mathrm{M}_{\mathrm{F}}}\right)\times \text{Sym}(n))$ for some $\mathrm{M}_{\mathrm{F}} \geq 0$. Further assume that
\begin{equation}\label{eq.elliwh}
   |\xi|\geq \mathrm{M}_{\mathrm{F}} \quad \text{and} \quad F_0(x, \xi, X)\leq 0 \quad \Longrightarrow \quad \mathscr{P}^-_{\lambda,\Lambda}(X)-\sigma(x)|\xi|-f_0(x)\leq 0.
\end{equation}
for continuous functions $f_0$ and $\sigma$ in $B_1$. Then, for any $q_1 > n$
\[
   \|u\|_{L^{p_0}\left(B_{\frac{1}{4}}\right)}\leq C.\left\{\inf_{B_{\frac{1}{2}}} u+\max\left\{\mathrm{M}_{\mathrm{F}}, \|f_0\|_{L^n(B_1)}\right\}\right\}
\]
for some $p_0(\verb"universal")>0$ and a constant $C>0$ depending on $n, q_1, \lambda, \Lambda$ and $\|\sigma\|_{L^{q_1}(B_1)}$.
\end{theorem}

\begin{theorem}[{\bf Local Maximum Principle, \cite[Theorem 3]{I11}}] \label{localmax}
Let $u$ be a continuous function satisfying
$$
   F_0(x, \nabla u, D^2u)\geq 0 \quad\textrm{ in }\quad B_1
$$
in the viscosity sense. Assume that $F_0$ is uniformly elliptic in the $X$ variable (see (A1) condition) and $F_0 \in C^0(B_1\times \left(\R^n \setminus B_{\mathrm{M}_{\mathrm{F}}}\right)\times \text{Sym}(n))$ for some $\mathrm{M}_{\mathrm{F}} \geq 0$. Further assume that
\begin{equation}\label{eq.elliplmp}
   |\xi|\geq \mathrm{M}_{\mathrm{F}} \quad \text{and} \quad F_0(x, \xi, X)\geq 0 \quad \Longrightarrow \quad \mathscr{P}^+_{\lambda,\Lambda}(X)+\sigma(x)|\xi|+f_0(x)\geq 0.
\end{equation}
for continuous functions $f_0$ and $\sigma$ in $B_1$. Then, for any $p_1>0$ and $q_1 >n$
\[
\sup_{B_{\frac{1}{4}}}u\leq C.\left\{\|u^+\|_{L^{p_1}\left(B_{\frac{1}{2}}\right)}+\max\left\{\mathrm{M}_{\mathrm{F}}, \|f_0\|_{L^n(B_1)}\right\}\right\}
\]
where $C>0$ is a constant depending on $n, q_1, \lambda, \Lambda,\|\sigma\|_{L^{q_1}(B_1)}$ and $p_1$.
\end{theorem}

Let us recall that such results were proved in Imbert's manuscript \cite{I11} by following the strategy of the uniformly elliptic case, see \cite[Section 4.2]{CC95}. Such a strategy is based on the so-called $L^\varepsilon$-Lemma, which establishes a polynomial decay for the measure of the super-level sets of a non-negative super-solution for the Pucci extremal operator $\mathscr{P}^+_{\lambda,\Lambda}$:
\begin{equation}\label{eq.lep}
\Leb(\left\{x \in B_1:u(x)>t\right\}\cap B_1)\leq \frac{C}{t^{\varepsilon}}.
\end{equation}

Unfortunately, Imbert's manuscript has a gap in the proof of \eqref{eq.lep}. Such an error was recently made up in a joint work with Silvestre, see \cite{IS16}, where an appropriate $L^\varepsilon$-estimate was addressed. In fact, their proof holds for ``Pucci extremal operators for large gradients'' defined, for a fixed $\tau$, by:
\[
\widetilde{\mathscr{P}}^+_{\lambda,\Lambda}(D^2u, \nabla u)\defeq
\left\{\begin{array}{ll}
\mathscr{P}^+_{\lambda,\Lambda}(D^2 u)+\Lambda|\nabla u| & \text{ if }|\nabla u|\geq \tau \\
+\infty & \text{ otherwise }
\end{array}\right.
\]
\[
\widetilde{\mathscr{P}}^-_{\lambda,\Lambda}(D^2u, \nabla u)\defeq
\left\{\begin{array}{ll}
\mathscr{P}^-_{\lambda,\Lambda}(D^2 u)-\Lambda|\nabla u| & \text{ if }|\nabla u|\geq \tau \\
-\infty & \text{ otherwise.}
\end{array}\right.
\]
The $L^\varepsilon$-estimate is proved to hold whenever $\tau\leq \varepsilon_0$ is universal (see, \cite[Theorem 5.1]{IS16}). Moreover, notice that the ellipticity condition $\widetilde{\mathscr{P}}^-_{\lambda,\Lambda}$ is consistent with \eqref{eq.elliwh} if we take $\sigma(x)\equiv\Lambda$. Precisely, if \eqref{eq.elliwh} and $u$ is a super-solution for $F_0$, then it is also a super-solution for $\widetilde{\mathscr{P}}^-_{\lambda,\Lambda}$ with right hand side $f_0$. An analogous reasoning is valid for $\widetilde{\mathscr{P}}^+_{\lambda,\Lambda}$ and \eqref{eq.elliplmp}.

In this point, once the $L^\varepsilon$ is derived, the proof of Theorem \ref{wharn} is exactly as the one in \cite{I11} which is, in turn, a modification of the uniformly elliptic case in \cite[Theorem 4.8, a]{CC95}. As for Theorem \ref{localmax}, it also follows from \eqref{eq.lep} by assuming (in a fist moment) that the $L^\varepsilon$ norm of $u^+$ is small and then obtaining the general result by interpolation. Indeed, the smallness of the $L^\varepsilon$ norm readily implies \eqref{eq.lep} which in turn gives $u$ is bounded (see, \cite[Lemma 4.4]{CC95}, which is adapted in \cite[Section 7.2]{I11}).

Notice that our class of operators fits in this scenario by setting
\[
   F_0(x, \nabla v, D^2v) \defeq \mathcal{H}(x, \nabla v)F(x, D^2 v)-f(x).
\]
and
\[
f_0(x)\defeq \frac{L_1^{-1}f^{+}(x)}{\varepsilon_0^p + \mathfrak{a}(x)\varepsilon_0^q} \quad \text{for suitable} \quad \varepsilon_0>0.
\]
In effect, we have that whenever
\[
\mathcal{H}(x, \nabla v + \xi)F(x, D^2v)\leq f(x) \quad \text{in} \quad B_1
\]
in the viscosity sense, then the ellipticity condition of $F$, this is (A1), ensures us that
\[
\mathscr{P}^-_{\lambda,\Lambda}(D^2v)\leq F(x, D^2v)\leq \frac{f(x)}{\mathcal{H}(x, \nabla v + \xi)}\le \frac{f^{+}(x)}{\mathcal{H}(x, \nabla v + \xi)} \quad \text{in} \quad B_1\cap \left\{|\nabla v + \xi|> \varepsilon_0\right\},
\]
whenever $|\nabla v|>\mathrm{M}_{\mathrm{F}} = \frac{3}{2}\varepsilon_0$ and $|\xi| \leq \frac{1}{2}\varepsilon_0$ so that
\[
\mathscr{P}^-_{\lambda,\Lambda}(D^2v)-\Lambda|\nabla v|-f_0(x)\leq \left(\frac{1}{\mathcal{H}(x, \nabla v + \xi)}-\frac{L_1^{-1}}{\varepsilon_0^p + \mathfrak{a}(x)\varepsilon_0^q}\right)f^{+}(x)\leq 0 \quad \text{in} \quad B_1\cap \left\{|\nabla v|> \frac{3}{2}\varepsilon_0\right\}.
\]

Remember that the constants obtained in \cite{IS16} are monotone with respect to $\tau$ and bounded away from zero and infinity, so we get a uniform estimate as \eqref{eq.lep} for supersolutions of $\mathcal{G}_{\xi}[v]\defeq \mathcal{H}(x, \nabla v + \xi)F(x, D^2v)$.

Therefore, in such a situation we have (recall $\sigma(x)\equiv\Lambda$) from Theorem \ref{wharn}:
 \begin{equation}\label{EqWHN}
      \|v\|_{L^{p_0}\left(B_{\frac{1}{4}}\right)}\leq\displaystyle C.\left\{\inf_{B_{\frac{1}{2}}} v + \frac{3}{2}\varepsilon_0 + \left\|f_0\right\|_{L^{N}(B_1)}\right\} \leq \Xi_0,
 \end{equation}
where
$$
\Xi_0 \defeq \left\{
 \begin{array}{lcl}
   \displaystyle C.\left\{\inf_{B_{\frac{1}{2}}} v + \min\left\{\frac{3}{2}, \left[(q+1)\sqrt[n]{|B_1|}L_1^{-1}\left\|\frac{f^{+}}{1+ \mathfrak{a}}\right\|_{L^{\infty}(B_1)}\right]^{\frac{1}{q+1}}\right\}\right\} & \text{if} & \varepsilon_0 \in (0, 1] \\
   \displaystyle C.\left\{\inf_{B_{\frac{1}{2}}} v + \min\left\{\frac{3}{2},\left[(p+1)\sqrt[n]{|B_1|}L_1^{-1}\left\|\frac{f^{+}}{1+ \mathfrak{a}}\right\|_{L^{\infty}(B_1)}\right]^{\frac{1}{p+1}}\right\}\right\} & \text{if} & \varepsilon_0 \in (1, \infty).
 \end{array}
 \right.
$$
Notice that we have used in above inequalities that the function
$$
(0, \infty) \ni t \mapsto \mathfrak{h}(t) = \frac{3}{2}t + \frac{1}{t^s} \left(\sqrt[n]{|B_1|}L_1^{-1}\left\|\frac{f^{+}}{1+ \mathfrak{a}}\right\|_{L^{\infty}(B_1)}\right)
$$
is optimized (lowest upper bound) when $t^{\ast} = \left(\frac{2s}{3}\sqrt[n]{|B_1|}L_1^{-1}\left\|\frac{f^{+}}{1+ \mathfrak{a}}\right\|_{L^{\infty}(B_1)}\right)^{\frac{1}{s+1}}$ for $s \in (0, \infty)$.

In conclusion, in any case, we obtain (since $0<p\le q<\infty$)
{\small{
$$
 \displaystyle \|v\|_{L^{p_0}\left(B_{\frac{1}{4}}\right)}\leq  C.\left\{\inf_{B_{\frac{1}{2}}} v + (q+1)^{\frac{1}{q+1}}\Pi^{f^{+}, \mathfrak{a}}_{p, q}\right\}
$$
}}
where
$$
\Pi^{f^+, \mathfrak{a}}_{p, q} \defeq \max\left\{\left[\sqrt[n]{|B_1|}L_1^{-1}\left\|\frac{f^{+}}{1+ \mathfrak{a}}\right\|_{L^{\infty}(B_1)}\right]^{\frac{1}{p+1}}, \left[\sqrt[n]{|B_1|}L_1^{-1}\left\|\frac{f^{+}}{1+ \mathfrak{a}}\right\|_{L^{\infty}(B_1)}\right]^{\frac{1}{q+1}}\right\}.
$$

In a similarly way, from Theorem \ref{localmax}, if
\[
\mathcal{H}(x, \nabla v + \xi)F(x, D^2v)\geq f(x) \quad \text{in} \quad B_1
\]
in the viscosity sense we again have for suitable $\varepsilon_0>0$
\[
\mathscr{P}^+_{\lambda,\Lambda}(D^2v)\geq F(x, D^2v) \geq \frac{f(x)}{\mathcal{H}(x, \nabla v + \xi)} \geq -\frac{f^{-}(x)}{\mathcal{H}(x, \nabla v + \xi)}  \quad \text{in} \quad B_1\cap \left\{|\nabla v + \xi|> \varepsilon_0\right\},
\]
whenever $|\nabla v|> \mathrm{M}_{\mathrm{F}} = \frac{3}{2}\varepsilon_0$ and $|\xi| \leq \frac{1}{2}\varepsilon_0$, we can set $f_0(x)\defeq \frac{L_1^{-1}f^{-}(x)}{\varepsilon_0^{p}+\mathfrak{a}(x)\varepsilon_0^{p}}
$
to get
\[
\mathscr{P}^+_{\lambda,\Lambda}(D^2v)+\Lambda|\nabla v|+f_0(x)\geq \left(\frac{L_1^{-1}}{\varepsilon_0^{p}+\mathfrak{a}(x)\varepsilon_0^{p}}-\frac{1}{\mathcal{H}(x, \nabla v + \xi)}\right)f^{-}(x) \geq 0  \quad \text{in} \quad B_1\cap \left\{|\nabla v|> \frac{3}{2}\varepsilon_0\right\}.
\]

Therefore, we have from Theorem \ref{localmax}:
\begin{equation}\label{EqLMP}
\begin{array}{ccl}
 \displaystyle \sup_{B_{\frac{1}{2}}} v & \leq & \displaystyle \|u^+\|_{L^{p_1}(B_1)}+ \frac{3}{2}\varepsilon_0+ \left\|f_0\right\|_{L^{n}(B_1)}\\
   & \le & \Xi_1,
\end{array}
\end{equation}
where, as before, we can estimate
$$
\Xi_1\defeq\left\{
 \begin{array}{lcl}
   \displaystyle C.\left\{\|v^+\|_{L^{p_1}(B_1)} + \min\left\{\frac{3}{2},\left[(q+1)\sqrt[n]{|B_1|}L_1^{-1}\left\|\frac{f^{-}}{1+ \mathfrak{a}}\right\|_{L^{\infty}(B_1)}\right]^{\frac{1}{q+1}}\right\}\right\} & \text{if} & \varepsilon_0 \in (0, 1] \\
   \displaystyle C.\left\{\|v^+\|_{L^{p_1}(B_1)} + \min\left\{\frac{3}{2}, \left[(p+1)\sqrt[n]{|B_1|}L_1^{-1}\left\|\frac{f^{-}}{1+ \mathfrak{a}}\right\|_{L^{\infty}(B_1)}\right]^{\frac{1}{p+1}}\right\}\right\} & \text{if} & \varepsilon_0 \in (1, \infty),
 \end{array}
 \right.
$$

Therefore, in any setting (since $0<p\le q<\infty$)
$$
\displaystyle \sup_{B_{\frac{1}{2}}} v \le C.\left\{\|v^+\|_{L^{p_1}(B_1)} + (q+1)^{\frac{1}{q+1}}\Pi^{f^-, \mathfrak{a}}_{p, q}\right\}
$$
thereby concluding this analysis.

Finally, combining \eqref{EqWHN} and \eqref{EqLMP}, we obtain the following Harnack inequality:

\begin{theorem}[{\bf Harnack inequality}]\label{ThmHarIneq} Let $u$ be a non-negative viscosity solution to
$$
F_0(x, \nabla v + \xi, D^2 v) = 0 \quad \text{in} \quad B_1.
$$
Then,
$$
\displaystyle \sup_{B_{\frac{1}{2}}} u(x) \leq \mathrm{C} \cdot\left\{\inf_{B_{\frac{1}{2}}} u(x) + (q+1)^{\frac{1}{q+1}}\Pi^{f, \mathfrak{a}}_{p, q}\right\},
$$
where $\mathrm{C} = \mathrm{C}(n, \lambda, \Lambda)>0$.
\end{theorem}

\begin{remark}[{\bf Harnack inequality - scaled version}]\label{HarnIneqSV} For our purposes, it will be useful to obtain a scaled version of Harnack inequality: let $v$ be a non-negative viscosity solution to
$$
F_0(x, \nabla v + \xi, D^2v) = 0 \quad \text{in} \quad B_{r} \quad \text{for a fixed} \quad r \in (0, \infty) \quad \quad \text{and any} \quad \xi \in \mathbb{R}^n,
$$
where (A0)-(A2), \eqref{1.2} and \eqref{1.3} are in force. Then,
$$
\displaystyle \sup_{B_{\frac{r}{2}}} v(x) \leq \mathrm{C} \cdot\left\{\inf_{B_{\frac{r}{2}}} v(x) + (q+1)^{\frac{1}{q+1}}\max\left\{r^{\frac{p+2}{p+1}}, r^{\frac{p+2}{q+1}}\right\}\Pi^{f, \mathfrak{a}}_{p, q}\right\},
$$
where $\mathrm{C} = \mathrm{C}(n, \lambda, \Lambda)>0$.
\end{remark}

\subsection{The Comparison Principle: Doubly degenerate scenario}

In this subsection, we prove the classical result known as Comparison Principle. In order to do so, we consider the following approximation problem
\begin{equation}\label{appro_problem}
\mathcal{G}_{\varepsilon}[u] \defeq	\mathcal{H}(x, \varepsilon+\nabla u)[\varepsilon u + F(x, D^2 u)]  =  f(x) \quad \text{ in } \quad \Omega
\end{equation}
for $0<\varepsilon<1$ and $u\in C^0(\overline{\Omega})$. The desired comparison result will hold true by letting $\varepsilon \to 0^{+}$. The main idea of the proof follows essentially from \cite{CIL92} (see also \cite{HPRS21} for more details). As a separate interesting point, such a result is a non-homogeneous degenerate counterpart for the degenerate one in \cite[Theorem 1.1]{BD04}.

\begin{theorem}[{\bf Comparison Principle}]\label{comparison principle}
	Assume that assumptions (A0)-(A1), \eqref{1.2} and \eqref{1.3} there hold. Let $f \in C^0(\overline{\Omega})$. Suppose $u_1$ and $u_2$ are respectively a viscosity supersolution and 	subsolution of \eqref{appro_problem}. If $u_1 \leq u_2$ on $\partial \Omega$, then $u_1 \leq u_2$ in $\Omega$.
\end{theorem}
\begin{proof}
	We shall prove this result by contradiction. For this end, suppose that such a statement is false. Then, we can assume that
	$$
	\mathrm{M}_0\coloneqq\max_{x\in\overline{\Omega}}(u_1-u_2)(x)>0.
	$$
	
	For $\delta>0$, let us define
	$$
	\mathrm{M}_{\delta}\coloneqq \max_{x,y\in\overline{\Omega}}\left[u_1(x)-u_2(y)-\frac{|x-y|^2}{2\delta}\right].
	$$
	Assume that the maximum $\mathrm{M}_\delta$ is attained in a point $(x_\delta, y_\delta)\in \overline{\Omega}\times\overline{\Omega}$ and notice that $\mathrm{M}_\delta\geq \mathrm{M}_0$.
	
	From \cite[Lemma 3.1]{CIL92}, we have that
	\begin{equation}\label{limit-CP}
	\lim_{\delta\to0}\frac{|x_\delta-y_\delta|}{\delta}=0.
	\end{equation}
	This implies that $x_\delta,y_\delta\in\Omega$ for $\delta$ sufficiently small. Moreover, \cite[Theorem 3.2 and Proposition 3 ]{CIL92} assures the existence of a limiting super-jet $\left(\frac{x_\delta-y_\delta}{\delta}, \mathrm{X}\right)$ of $u_1$ at $x_\delta$ and a limiting sub-jet $\left(\frac{x_\delta-y_\delta}{\delta}, \mathrm{Y}\right)$ of $u_2$ at $y_\delta$, where the matrices $\mathrm{X}$ and $\mathrm{Y}$ satisfies the inequality
	\begin{equation}\label{matrices-ine}
		-\frac{3}{\delta}
		\begin{pmatrix}
			\mathrm{Id}_n& 0 \\
			0& \mathrm{Id}_n
		\end{pmatrix}
		\leq
		\begin{pmatrix}
			\mathrm{X}&0\\
			0&-\mathrm{Y}
		\end{pmatrix}
		\leq
		\frac{3}{\delta}
		\begin{pmatrix}
			\mathrm{Id}_n & -\mathrm{Id}_n \\
			-\mathrm{Id}_n& \mathrm{Id}_n
		\end{pmatrix},	
	\end{equation}
where $\mathrm{Id}_n$ is the identity matrix.

Observe that the assumption (A1) implies that the operator $F$ is, in particular, degenerate elliptic. It implies that,
\begin{equation}\label{MonotF}
  F(x, \mathrm{X})\geq F(x, \mathrm{Y})
\end{equation}
for every $x\in\Omega$ fixed, since $\mathrm{X}\leq \mathrm{Y}$ from \eqref{matrices-ine}.

Therefore, as a consequence of the two viscosity inequalities (and sentence \eqref{MonotF})
$$
	\mathcal{H}\left(x_\delta, \varepsilon+\frac{|x_\delta-y_\delta|}{\delta}\right)[\varepsilon u_1(x_\delta) + F(x_\delta, \mathrm{X})]  \leq  f(x_\delta)
$$
and
$$
\mathcal{H}\left(y_\delta, \varepsilon+\frac{|x_\delta-y_\delta|}{\delta}\right)[\varepsilon u_2(y_\delta) + F(y_\delta, \mathrm{Y})]  \geq  f(y_\delta)
$$
we get
\begin{equation}\label{estl-CP}
  \begin{array}{rcl}
    \frac{\varepsilon \mathrm{M}_0}{2} & \le & \varepsilon (u_1(x_\delta)-u_2(y_\delta)) \\
     & \le & \frac{f(x_\delta)}{\mathcal{H}\left(x_\delta, \varepsilon+\frac{|x_\delta-y_\delta|}{\delta}\right)}+[F(x_\delta, \mathrm{Y})-F(y_\delta, \mathrm{Y})]-\frac{f(y_\delta)}{\mathcal{H}\left(y_\delta, \varepsilon+\frac{|x_\delta-y_\delta|}{\delta}\right)}
  \end{array}
\end{equation}

Now, observe that, from assumption (A2), we can estimate
\begin{equation}\label{est2-CP}
|F(x_\delta, \mathrm{Y})-F(y_\delta, \mathrm{Y})|\leq \mathrm{C}_{\mathrm{F}}\omega(|x_\delta-y_\delta|)\|\mathrm{Y}\|.
\end{equation}
Moreover, if $\omega_f$ is a modulus of continuity of $f$ on $\overline{\Omega}$, from assumptions \eqref{1.2} and \eqref{Cont-H} we obtain

\begin{equation}\label{est3-CP}
	\begin{aligned}
		&\frac{f(x_\delta)}{\mathcal{H}\left(x_\delta, \varepsilon+\frac{|x_\delta-y_\delta|}{\delta}\right)}-\frac{f(y_\delta)}{\mathcal{H}\left(y_\delta, \varepsilon+\frac{|x_\delta-y_\delta|}{\delta}\right)}\\
		\leq&\frac{f(x_\delta)-f(y_\delta)}{\mathcal{H}\left(x_\delta, \varepsilon+\frac{|x_\delta-y_\delta|}{\delta}\right)}+f(y_\delta)\left[\frac{1}{\mathcal{H}\left(x_\delta, \varepsilon+\frac{|x_\delta-y_\delta|}{\delta}\right)}-\frac{1}{\mathcal{H}\left(y_\delta, \varepsilon+\frac{|x_\delta-y_\delta|}{\delta}\right)}\right]\\
		\leq&\frac{\omega_f(|x_\delta-y_\delta|)}{L_1 \cdot \mathcal{K}_{p, q, \mathfrak{a}}\left(x_\delta, \left|\varepsilon+\frac{|x_\delta-y_{\delta}|}{\delta}\right|\right)}\\
		+&\|f\|_{L^{\infty}(\Omega)}\frac{\left|\mathcal{H}\left(y_\delta, \varepsilon+\frac{|x_\delta-y_\delta|}{\delta}\right)-\mathcal{H}\left(x_\delta, \varepsilon+\frac{|x_\delta-y_\delta|}{\delta}\right)\right|}{L^2_1 \cdot \mathcal{K}_{p, q, \mathfrak{a}}\left(x_\delta, \left|\varepsilon+\frac{|x_\delta-y_{\delta}|}{\delta}\right|\right)\mathcal{K}_{p, q, \mathfrak{a}}\left(y_\delta, \left|\varepsilon+\frac{|x_\delta-y_{\delta}|}{\delta}\right|\right)}\\
		\leq&\frac{\omega_f(|x_\delta-y_\delta|)}{L_1 \left|\varepsilon+\frac{|x_\delta-y_{\delta}|}{\delta}\right|^p} + \mathrm{C}_{\mathfrak{a}}\|f\|_{L^{\infty}(\Omega)}\frac{\omega_{\mathfrak{a}}(|x_{\delta}-y_{\delta}|)\left|\varepsilon+\frac{|x_\delta-y_{\delta}|}{\delta}\right|^q}{L^2_1 \left|\varepsilon+\frac{|x_\delta-y_{\delta}|}{\delta}\right|^{2p}}
	\end{aligned}
\end{equation}

In conclusion, by combining the sentences \eqref{estl-CP}, \eqref{est2-CP} and \eqref{est3-CP} we reach
$$
\begin{array}{rcl}
  \frac{\varepsilon \mathrm{M}_0}{2} & \le & \mathrm{C}_{\mathrm{F}}\omega(|x_\delta-y_\delta|)\|\mathrm{Y}\| + \frac{\omega_f(|x_\delta-y_\delta|)}{L_1 \left|\varepsilon+\frac{|x_\delta-y_{\delta}|}{\delta}\right|^p} + \mathrm{C}_{\mathfrak{a}}\|f\|_{L^{\infty}(\Omega)}\frac{\omega_{\mathfrak{a}}(|x_{\delta}-y_{\delta}|)\left|\varepsilon+\frac{|x_\delta-y_{\delta}|}{\delta}\right|^q}{L^2_1 \left|\varepsilon+\frac{|x_\delta-y_{\delta}|}{\delta}\right|^{2p}} \\
   & = & \text{o}(1) \quad \text{as}\quad \delta \to 0^{+},
\end{array}
$$
which yields a contradiction, thereby proving the Comparison Principle by letting $\varepsilon \to 0^{+}$.
\end{proof}

\subsection*{Acknowledgments}

\hspace{0.65cm} J.V. da Silva and G.C. Ricarte have been partially supported by Conselho Nacional de Desenvolvimento Cient\'{i}fico e Tecnol\'{o}gico (CNPq-Brazil) under Grants No. 310303/2019-2 and No. 303078/2018-9. G.C. Rampasso is partially supported by CAPES-Brazil. This research was financed in part by the Coordena\c{c}\~{a}o de Aperfei\c{c}oamento de Pessoal de N\'{i}vel Superior - (CAPES - Brazil) - Finance Code 001. H. Vivas was partially supported by Consejo Nacional de Investigaciones Cient\'{i}ficas y T\'{e}cnicas (CONICET-Argentina).

%%%%%%%%%%%%%%%%%%%%%%%%%%%%%%%%%
% Bibliography
%%%%%%%%%%%%%%%%%%%%%%%%%%%%%%%%%

\end{document}